\documentclass[11pt,a4paper,reqno]{amsart}
\usepackage[english]{babel}
\usepackage[applemac]{inputenc}
\usepackage[T1]{fontenc}
\usepackage{palatino}
\usepackage{amsmath}
\usepackage{amssymb}
\usepackage{amsthm}
\usepackage{amsfonts}
\usepackage{graphicx}
\usepackage{color}
\usepackage{mathtools}

\usepackage{vmargin}

\setmarginsrb{23mm}{20mm}{23mm}{20mm}{10mm}{10mm}{10mm}{10mm}

\usepackage[colorlinks = true, citecolor = black]{hyperref}
\newtheorem{theorem}{Theorem}[section]
\newtheorem{obs}{Observation}[section]

\newtheorem{cor}[theorem]{Corollary}
\newtheorem{prop}[theorem]{Proposition}

\newtheorem{lem}[theorem]{Lemma}

\theoremstyle{definition}
\newtheorem{defn}{Definition}[section]

\newtheorem{remark}[theorem]{Remark}

\def\vint{\mathop{\mathchoice%
          {\setbox0\hbox{$\displaystyle\intop$}\kern 0.22\wd0%
           \vcenter{\hrule width 0.6\wd0}\kern -0.82\wd0}%
          {\setbox0\hbox{$\textstyle\intop$}\kern 0.2\wd0%
           \vcenter{\hrule width 0.6\wd0}\kern -0.8\wd0}%
          {\setbox0\hbox{$\scriptstyle\intop$}\kern 0.2\wd0%
           \vcenter{\hrule width 0.6\wd0}\kern -0.8\wd0}%
          {\setbox0\hbox{$\scriptscriptstyle\intop$}\kern 0.2\wd0%
           \vcenter{\hrule width 0.6\wd0}\kern -0.8\wd0}}%
          \mathopen{}\int}

\newcommand{\ep}{\epsilon}
\newcommand{\Om}{\Omega}
\newcommand{\om}{\omega}

\newcommand{\R}{\mathbb{R}}
\newcommand{\Rn}{{\mathbb R}^n}

\newcommand{\ud}{\mathrm {d}}
\newcommand{\Hp}{\mathcal{H}^p}
\newcommand{\dist}{\operatorname{dist}}

\newcommand{\Sn}{\mathbb{S}^{n-1}}
\newcommand{\A}{\mathcal{A}}

\definecolor{blau}{rgb}{0.1,0.0,0.9}
\definecolor{violet}{rgb}{0.54, 0.17, 0.89}
\newcommand{\blue}{\color{blau}}

\newcommand{\kom}[1]{}
\renewcommand{\kom}[1]{{\bf \blue /#1/}}

\newcounter{komcounter}
\numberwithin{komcounter}{section}

\begin{document}

\title[Hardy spaces and QR mappings]{Hardy spaces and quasiregular mappings}
\author[T.\ Adamowicz]{Tomasz Adamowicz{\small$^1$}}
\address{The Institute of Mathematics, Polish Academy of Sciences \\ ul. \'Sniadeckich 8, 00-656 Warsaw, Poland}
\email{tadamowi@impan.pl}
\author[M.\ J.\ Gonz\'alez ]{Mar\'ia\ J.\ Gonz\'alez{\small$^2$}}
\address{Departamento de Matem\'aticas, Universidad de C\'adiz, 11510 Puerto Real (C\'adiz), Spain}
\email{majose.gonzalez@uca.es}

\thanks{{\small$^1$}T.A.\ was supported by the National Science Center, Poland (NCN), UMO-2017/25/B/ST1/01955. 
\newline
{\small$^2$} M.J.G. was  supported in part by the Spanish Ministerio de Ciencia e Innovaci\'on (grant no.  PID2021-123151NB-I00), and  by the grant "Operator Theory: an interdisciplinary approach", reference ProyExcel\_00780, a project financed in the 2021 call for Grants for Excellence Projects, under a competitive bidding regime, aimed at entities qualified as Agents of the Andalusian Knowledge System, in the scope of the Andalusian Research, Development and Innovation Plan (PAIDI 2020). Counseling of University, Research and Innovation of the Junta de Andaluc\'ia.}
\keywords{Carleson measure, Hardy spaces, multiplicity, non-tangential maximal function, quasiregular mappings}
%\date{\today}
\subjclass[2010]{(Primary) 30C65  (Secondary) 30H10}

\begin{abstract}
We study Hardy spaces $\mathcal{H}^p$, $0<p<\infty$ for quasiregular mappings on the unit ball $B$ in $\Rn$ 
which satisfy appropriate growth and multiplicity conditions. Under these conditions we recover several classical results for analytic functions and quasiconformal mappings in $\mathcal{H}^p$. In particular, we characterize $\mathcal{H}^p$ in terms of non-tangential limit functions and non-tangential maximal functions of quasiregular mappings. Among applications we show that every quasiregular map in our class belongs to $\mathcal{H}^p$ for some $p=p(n,K)$. Moreover, we provide characterization of Carleson measures on $B$ via integral inequalities for quasiregular mappings on $B$. We also discuss the Bergman spaces of quasiregular mappings and their relations to $\mathcal{H}^p$ spaces and analyze correspondence between results for $\mathcal{H}^p$ spaces and $\A$-harmonic functions. 

A key difference between the previously known results for quasiconformal mappings in $\Rn$ and our setting is the role of multiplicity conditions and the growth of mappings that need not be injective.

Our paper extends results by Astala and Koskela, Jerison and Weitsman, Jones, Nolder, and Zinsmeister.
\end{abstract}

\maketitle

\section{Introduction}	

We say that a mapping $f$ from a unit ball $B(0,1)\subset \Rn$ with values in $\Rn$ belongs to \emph{the Hardy space $\mathcal{H}^p$} for a $0<p<\infty $, if
\[
\sup_{0<r<1} \left(\vint_{\Sn} |f (r\om)|^p ~\ud\om\right)^{\frac1p} := \|f\|_{\mathcal{H}^p}< \infty.
\]
In the setting of holomorphic functions on unit disc this definition is well known and it is due to Hardy and Littlewood \cite{hl}, that a holomorphic function $f\in \mathcal{H}^p$ if and only if the non-tangential maximal function of $f$ belongs to $L^p(\mathbb{S}^1)$. 
The Hardy spaces play a significant role in complex analysis, for instance in the studies of the boundary behaviour of functions, in the theory of Carleson measures and in the Nevanlinna theory, see e.g.~\cite{D, G}. Furthermore, the Hardy spaces have met geometric mapping theory, leading to the $H^p$ theory for quasiconformal mappings in Euclidean spaces, see Astala--Koskela~\cite{AK}, Zinsmeister~\cite{z} and Nolder~\cite{nol}, also in the setting of the Heisenberg group $\mathbb{H}_1$, see~\cite{af}.  

The main goal of this paper is to introduce an $\mathcal{H}^p$ theory for quasiregular mappings on the unit ball in $\Rn$ for $n\geq 2$, thus extending both the quasiconformal $\mathcal{H}^p$ theory to the setting of not necessarily injective mappings, as well as the classical $\mathcal{H}^p$ theory for analytic functions (recall that the $1$-quasiregular mappings in the plane coincide with analytic functions). The first step toward developing this program has been made in~\cite{AG}, where the $\mathcal{H}^p$ theory for quasiregular mappings in the plane is studied. However, the case $n=2$ is the only one in which the Stoilow factorization is available and in higher dimensions a different approach is needed. It turns out that the following two conditions are crucial:
\begin{itemize}
\item[(1)] the control of the multiplicity of a mapping, see condition~\eqref{cond-m} and Section 3 below;
\item[(2)] the growth of the mapping, see the discussion on the Miniowitz estimates~\eqref{est-min} and Theorem~\ref{thm-min}. 
\end{itemize}
The multiplicity and the growth conditions affect existence of non-tangential limits and integrability of the Jacobian of the underlying quasiregular map, also allow the Carleson measures estimates.

Our main result is the following characterization of the Hardy spaces of quasiregular mappings:

\begin{theorem}\label{thm-main}
  Let $n \geq 2$ and $f:B\to \Rn \setminus \{0\}$ be a $K$-quasiregular mapping in class~\eqref{est-min} satisfying the multiplicity condition~\eqref{cond-m} with $0\leq a<n-1$. Then the following conditions are equivalent for every $p>0$:
  \begin{itemize}
   \item[(1)] $f\in \mathcal{H}^p$,
   \item[(2)] the non-tangential limit function $\tilde{f}\in L^p(\mathbb{S}^{n-1})$,
   \item[(3)] the non-tangential maximal function $f^*\in L^p(\mathbb{S}^{n-1})$.
  \end{itemize}
  Moreover, all the norms in conditions (1)-(3) are equivalent:
    \[
  \|f\|_{\mathcal{H}^p}\approx \|f^*\|_{L^p(\mathbb{S}^{n-1})} \approx \|\tilde{f}\|_{L^p(\mathbb{S}^{n-1})},
  \]
  where the equivalence constants depend only on $n, p $ and $a, K, \alpha$ (see~\eqref{est-min} for the definition of $\alpha$).
\end{theorem}

We present the proof of Theorem~\ref{thm-main} in Section~\ref{section-4} along several applications. The key auxiliary tool, which is of the independent interest, is the so-called Jones lemma, here obtained for the quasiregular mappings, see Lemma~\ref{lem-Jones-qr}. It asserts that under assumptions of Theorem~\ref{thm-main}, a quasiregular map $f$ defines the following family of Carleson measures on $B$ for $0<q<n$:
 \[
  \ud\mu=\frac{|Df(x)|^q}{|f(x)|^q}(1-|x|)^{q-1} \ud x. 
 \]
 Furthermore, the Carleson constant of $\mu$ depends only on $n, q$ and $K, \alpha, a$. Such a result extends the corresponding observation for quasiconformal mappings due to Jones~\cite[Lemma 4.2]{j} for $n=2$ and Astala--Koskela~\cite[Lemma 5.6]{AK} for $n\geq 2$. Let us briefly advertise some applications of Lemma~\ref{lem-Jones-qr} and Theorem~\ref{thm-main}, see details in Sections~\ref{section-3} and~\ref{section-4}:
\smallskip
 
\noindent
(1)\, In Corollary~\ref{cor: jer-weits} we prove a quasiregular counterpart of the Jerison--Weitsman theorem, cf. ~\cite[Theorem 1]{JW}, showing that under the assumptions of Theorem~\ref{thm-main} every quasiregular mapping belongs to some Hardy space $\mathcal{H}^p$ for $p=p(n,K)$.

\noindent
(2)\, Proposition~\ref{obs:cor45} provides a new characterization of Carleson measures in terms of quasiregular mappings. 

\noindent
(3)\, Proposition~\ref{prop-cor41} gives the submean value property for the modulus of a quasiregular mapping.

\noindent
(4)\, Theorem~\ref{t: Ap-spaces} gives the relation between the Hardy spaces $\mathcal{H}^p$ and the Bergman spaces of quasiregular mappings.

\noindent
(5)\, In last section of the paper we employ the relations between quasiregular mappings and elliptic PDEs to discuss how the result for $\mathcal{H}^p$ spaces reflect for related $\A$-harmonic functions, in particular for $p$-harmonic functions in the plane.
%\end{itemize}

\section{Preliminaries}	

 \subsection{Quasiregular mappings.}\label{sec-22} Let $\Om\subset \Rn$ be a domain. A continuous mapping $f:\Om \rightarrow \mathbb{R}^n$  is called $K$-\emph{quasiregular} for $K\geq 1$, if $f$ belongs to the Sobolev space $W_{loc}^{1,n}(\Om, \Rn)$ and the distortion inequality
  \[
  |Df(x)|^n\leq  K J_f (x)
  \] 
  holds for almost every $x\in \Om$. Here $|Df(x)|$ denotes the operator norm of the formal derivative of $f$ at $x\in \Om$ and $J_f(x)$ stands for the Jacobian determinant of $Df$ at $x$.

 If in addition we require $f$ to be a homeomorphism, then we say that $f$ is $K$-\textit{quasi\-con\-formal}. For comprehensive introductions to the topic and further references we refer, for instance, to~\cite{Ric} and Chapter 14 in~\cite{hkm} and \cite{resh} for results in $n\geq 2$ and~\cite{aim} for $n=2$. 
 
 Let us comment that the quasiregular mappings can be defined in various other equivalent ways, e.g. via the modulus of curve families. We will address that definition in the course of the presentation, see the discussion following Remark~\ref{rem-harnack}.

 \subsection{$\A$-harmonic functions.} 	 Since in what follows the interplay between the nonlinear potential theory and the quasiregular mappings will play a role, we now briefly recall the notion of $\mathcal{A}$-harmonic functions. Our presentation and notation is based on \cite[Chapter 3]{hkm} and we refer readers to this classical source for an excellent introduction to the topic, also to the description of relation of $\A$-harmonicity and quasiregular mappings, see~\cite[Chapter 14]{hkm}.
 
 Let $\A:\R^n\times \R^n\to \R^n$ be a Carath\'eodory mapping such that for some constants $0<\alpha\leq \beta<\infty$ and all $\xi\in\Rn$ and a.e. $x\in \Rn$ the following conditions are satisfied:
%\begin{itemize}
\begin{align*}
 &(1)\quad \langle \A(x,\xi), \xi \rangle \geq \alpha|\xi|^p,\qquad |\A(x,\xi)| \leq \beta |\xi|^{p-1} \qquad\hbox{(elliptic degeneracy and continuity)} \\
  &(2)\quad \langle \A(x,\xi)- \A(x,\eta), \xi-\eta \rangle >0 \hbox{ for } \xi,\eta\in \Rn, \xi\not=\eta \qquad\hbox{(monotonicity)}\\ 
&(3)\quad \A(x,t\xi)=t|t|^{p-2}\A(x,\xi),\quad t\in \R, t\not=0\qquad\hbox{($(p-1)$-homogeneity).}
\end{align*}
 %\end{itemize} 
 Related to $\A$-harmonic operators is the following class of elliptic PDEs.
 \begin{defn}\label{def:A-harm}
 Let $\Om\subset \Rn$ be a domain. A function $u\in W^{1,p}_{loc}(\Om, \R)$ is called an \emph{$\A$-harmonic function}, if it is a continuous (weak) solution to the following equation
 \[
  -{\rm div} \A(x,\nabla u)=0\quad\hbox{ in }\Om,
 \]
 meaning that
 \[
  \int_{\Om} \langle \A(x,\nabla u(x)), \nabla \phi(x) \rangle\, \ud x=0,\quad \phi\in C_0^\infty(\Om).
 \]
  \end{defn}
  
 The class of $\A$-harmonic functions encompasses the harmonic ones for $\A(x,\xi)=\xi$, as well as the $p$-harmonic functions, obtained for $\A(x,\xi)=|\xi|^{p-2}\xi$ and $1\leq p <\infty$. In what follows the case $p=n$ will be the most important from our point of view.
 	
 One of the key properties of the harmonic functions in Euclidean spaces is the mean value property and the corresponding inequalities holding for sub- and superharmonic functions. Similar inequalities are known for nonlinear degenerate elliptic PDEs such as, for instance, the $p$-harmonic or the $\mathcal{A}$-harmonic equations and we refer to~\cite{hkm} for a detailed discussion of this topic, especially to Chapters 3, 6 and 14 therein.  In particular, the following property holds.
 
\begin{lem}[cf. Theorem 3.34 in~\cite{hkm}]
 Let $u: B\to \R$ be an $\mathcal{A}$-harmonic function and $q>0$. Then, there exists a constant $C$ such that for all balls $B(x,r)\subset B$ 
 \begin{equation}\label{sup-est}
  \sup_{B(x,\frac12 r)}|u|\leq C\left(\vint_{B(x,r)} |u|^q\right)^\frac1q.
 \end{equation}
 The constant $C$ depends on $q$, $n$ and the structure conditions for an operator $\mathcal{A}$.
\end{lem}	
	
If $f:B\to\R^n$ is $K$-quasiregular, then its component functions satisfy an $\A$-harmonic equation given by the pull-back of the $n$-harmonic operator under $f$, see Theorem 14.42 in~\cite{hkm}.  Thus also the above supremum estimate holds with constant $C$ depending on $q$, $n$ and $K$, see Lemma 14.38 in~\cite{hkm} and also~\cite{km} for a related discussion.	
	
\subsection{Conformal automorphism of the Euclidean unit ball.} In several results below we use $T_x$, the self-conformal map of a unit ball in $\R^n$, such that $T_x(x)=0$ for a fixed $x\in B$, $x\not=0$. Properties of such map have been substantially used in~\cite{AK, z} and are well explained in a book~\cite{ahl}, see Formula (24) on pg. 25 in~\cite{ahl} and onwards. Below we use~\cite{ahl} to prove the following properties of $T_x$. Set 
\begin{equation}\label{def:Tx}
T_x(y):=\frac{(1-|x|^2)(y-x)-|y-x|^2x}{|x|^2\left|y-\frac{x}{|x|^2}\right|^2},\quad y\in B,\quad x\not=0.
\end{equation}	

The proof of the following lemma is presented in Appendix.
\begin{lem}\label{lem:Tx}
 The following properties hold for the above defined map $T_x$ for any $\om\in \partial B$, $0<r<\frac12$ and 
 $x\in B\cap B(\om, r)$ such that $x$ lies on the radial segment joining the origin with $\om$ and $|x-\om|=\frac{r}{2}$.
 \begin{itemize}
 \item[(1)] $J_{T_{x}}\approx \frac{1}{r^n}$,\quad on $B\cap B(\om,r)$,
 \item[(2)] $\frac{1-|T_{x}(y)|}{1-|y|} \approx \frac{1}{r}$,\quad for $y\in B\cap B(\om,r)$.
 \end{itemize}
 Moreover, if $0<k<\sqrt{2}-1$ and $x\in B$, then 
 $$
  B\left(0, \frac{k-k^2}{(2+k)^2}\right)\subset T_{x}\Big(B(x,k(1-|x|)\Big)\subset B\left(0,k^2+ 2k\right).
 $$
 Furthermore, since $T_x^{-1}=T_{-x}$, the analogous inclusions hold for the inverse map $T_x^{-1}$.
\end{lem}	

\section{Hardy spaces and geometry of quasiregular mappings}\label{section-3}	
%\komT{To be remembered: Let $f$ be QR and either has radial limits a.e. on $\Sn$ or satisfy the \"Akkinen-type condition (implying existence of radial limits a.e.). Does it imply that $f\in \Hp$? Ans: Thm 3.1
%}
%
%\komT{ Knowing that the radial limits exist allow us to discuss the radial maximal function and its relation to non-tangential maximal function, see Pavlovic's book, esp. Prop. 7.1.3 and 7.1.8.
%}
%
%\komT{It would be nice to have nontrivial examples of QR maps in $\Hp$, i.e. neither complex functions nor quasiconformal mappings. For example, the QR-version of the Jerison-Weitsman theorem. Perhaps \eqref{cond-m} and ~\eqref{cond-g} imply that result for some range of $p$? Ans: Cor. 2.7}
%
% \komT{The branch set for QR maps is not that well understood beyond the plane and there are results showing that the local index can be arbitrarily large for QR along fractal type sets, see e.g. Heinonen--Rickman 1998, Rickman--Srebro 1986, Rickman-book pg. 76}.

The purpose of this section is to define the Hardy spaces of mappings and discuss some basic interplay between quasiregular mappings and those spaces. This involves notions of multiplicity function and, naturally, the non-tangential limits. Moreover, we introduce some key growth properties for mappings, see the Miniowitz estimates in Theorem~\ref{thm-min} and its consequences, for instance the Harnack inequality, see Corollary~\ref{cor:harnack} and also Remark~\ref{rem-harnack}. However, the main result of this section is the quasiregular variant of the celebrated Jones lemma, see Lemma~\ref{lem-Jones-qr}, which provides us with the class of Carleson measures defined via the quasiregular maps. Moreover, the lemma is important tool in showing the main theorem, see Theorem~\ref{thm-main}.

\begin{defn}	\label{def-Hardy}
 Let $f:B(0,1)\to \Rn$ be a mapping. We say that \emph{$f$ belongs to the Hardy space $\mathcal{H}^p$} for a $0<p<\infty $, if
\[
\sup_{0<r<1} \left(\vint_{\Sn} |f (r\om)|^p ~\ud\om\right)^{\frac1p} := \|f\|_{\mathcal{H}^p}< \infty.
\]
\end{defn}	
	
Below we often abbreviate notation and write $B:=B(0,1)$.

We next formulate two key conditions on the multiplicity and the growth of quasiregular mappings in subject.

First, recall that by following the discussion in~\cite{hkm} (see pg. 254), for a subset $E\subset B$ we denote by $N(y, f, E)$ \emph{the crude multiplicity} of mapping $f$ defined as follows: $N(y, f, E):=\#\{x\in E: f(x)=y\}$. Similarly, we define 
\emph{the multiplicity function} of $f$:
\[
N(f, E)=\sup_{y\in \R^n} N(y, f, E).
\]
One of the key factors allowing the studies of quasiregular mappings, especially if the target space is unbounded, is the
multiplicity condition imposed on function $N$. For example, in~\cite[Theorem 4.1]{kmv}, see also condition (3) in~\cite{Ak}, the authors assume that a quasiregular mapping $f:B\to \Rn$ satisfies
\begin{equation}\label{precond-m}
 N(f, B(0,r))\leq \frac{C}{(1-r)^a}\quad \hbox{for all } 0<r<1,
\end{equation}
 for some constant $C>0$ and an exponent $a\in [0, n-1)$.

In what follows we consider the stronger condition holding for all hyperbolic balls in $B$, not only those centered at zero. 
\begin{defn}
 We say that a quasiregular mapping $f:B\to \Rn$ satisfies \emph{the multiplicity condition (M)}, if there exist constant $C>0$ and an exponent $a\in [0, n-1)$, such  that
\begin{equation}\label{cond-m}\tag{M}
\sup_{x\in B} N(f\circ T_x^{-1}, B(0, r))\leq \frac{C}{(1-r)^a}\quad \hbox{for all } 0<r<1.
\end{equation}
\end{defn}
In particular, for $x=0$ we retrieve condition~\eqref{precond-m}. Such a M\"obius invariant condition is imposed in order to ensure that the multiplicity of quasiregular maps $g:=f\circ T_x^{-1}$ is similar in nature to multiplicity of $f$.
 
Similarly to the discussion in~\cite{kmv} we impose the following \emph{growth condition (G)} on a quasiregular mapping $f:B\to \Rn$ with a constant $C>0$  and exponent $0<\beta<\infty$
\begin{equation}\label{cond-g}\tag{G}
 |f(x)|\leq \frac{C}{(1-|x|)^{\beta}}.
\end{equation}

The role of the exponent $\beta$ is explained in the following observation showing that a quasiregular map in Hardy spaces has the growth~\eqref{cond-g}. Notice that in the special planar case of $n=2$ and analytic $f$ we retrieve the growth exponent $\frac1p$ as in formula (3.9) in~\cite[Theorem 3.5]{G}.

\begin{obs}\label{thm1}
 Let $f:B\to \Rn$ be $K$-quasiregular and in Hardy space $\Hp$. Then $f$ satisfies the growth condition~\eqref{cond-g}. Namely, it holds that
\begin{equation}\label{cond-g-hp}
 |f(x)|^p\leq C\frac{\|f\|_{\Hp}^p}{(1-|x|)^{n-1}},\quad x\in B,
\end{equation}
and the constant $C$ depends on $p, n$ and $K$.
\end{obs}

\begin{proof}
 The supremum estimate~\eqref{sup-est} applied to coordinate functions of $f$ on hyperbolic balls for $q=p$ gives that
 \[
  |f_i(x)|\leq C \left(\vint_{B\left(x,\frac12(1-|x|)\right)} |f_i|^p\right)^\frac1p, \quad x\in B.
 \]
Hence, 
\[
 |f(x)|^p\leq \frac{C}{(1-|x|)^n}\int_{B\left(x,\frac12(1-|x|)\right)} |f|^p.
\]
Since the hyperbolic ball is contained in an annulus
\[
B\Big(x,\frac12(1-|x|)\Big)\subset B\Big(0, \frac12(3|x|-1)\Big)\setminus B\Big(0,\frac12(1+|x|)\Big) 
\]
 for $x$ such that $|x|>\frac13$, it holds that
\begin{align*}
|f(x)|^p\leq \frac{C}{(1-|x|)^n}\int_{B\left(x,\frac12(1-|x|)\right)} |f|^p\leq \frac{C}{(1-|x|)^n}\int_{B\left(0, \frac12(3|x|-1)\right)\setminus B\left(0,\frac12(1+|x|)\right)} |f|^p\leq \frac{C}{(1-|x|)^{n-1}} \|f\|_{\Hp}^p,
\end{align*}
where the last inequality is the consequence of the Fubini theorem and our definition of the Hardy space.

If $|x|\leq \frac13$, then trivially $(1-|x|)^n\geq \frac23 (1-|x|)^{n-1}$.  Moreover, $\|f\|_{L^p(B)}\leq \|f\|_{\Hp}$ by the Fubini theorem and  the assertion of the theorem holds as well, since then we have
\[
 |f(x)|^p\leq \frac{C}{(1-|x|)^n}\int_{B\left(x,\frac12(1-|x|)\right)} |f|^p\leq \frac{\frac32 C}{(1-|x|)^{n-1}} \|f\|_{L^p(B)}^p.
\] \end{proof}

The following simple observation shows that for a mapping in the Hardy space $\mathcal{H}^p$ if its radial limit function exists a.e. on the boundary of the ball, then such function is $L^p$-integrable. In what follows we employ Proposition~\ref{prop1} twice: in the proof of Corollary~\ref{cor-rad-lim} and in the proof of Theorem~\ref{thm-main}.

We define the \emph{non-tangential (radial) limit function} $\tilde{f}$ of a mapping $f$ as follows. Let $f:B\to \Rn$, then
\[
 \tilde{f}(\om):=\lim_{r\to 1} f(r\om) \quad\hbox{for }\om\in \Sn,
\]
whenever this limit exists.

\begin{prop}\label{prop1}
 Let $f\in \Hp$ and suppose that $f$ has radial limits a.e. in $\Sn$. Then $\tilde{f}\in L^p(\Sn)$.
\end{prop}
\begin{proof}
 By the Fatou lemma we have
 \begin{align*}
  \int_{\Sn} |\tilde{f}(\om)|^p\ud\om=  \int_{\Sn} \lim_{r\to 1}|f(r\om)|^p\ud\om &\leq \liminf_{r\to 1}\int_{\Sn} |f(r\om)|^p\ud\om \\
&  \leq \om_{n-1} \sup _{0<r<1} \vint_{\Sn} |f (r\om)|^p ~\ud\om =\om_{n-1} \|f\|_{\mathcal{H}^p}^p<\infty.
 \end{align*}
\end{proof}

Next, we recall one of the key auxiliary results providing conditions for a quasiregular map to have the boundary values. The importance of the following theorem is, perhaps, more clear upon recalling that even for a planar bounded quasiregular mappings it may happen that the radial limits exist only in a set of arbitrarily small Hausdorff dimension (see pg. 119-120 in~\cite{N}).

\begin{theorem}[see Theorem 4.1 in~\cite{kmv}]\label{thm-kmv}
   Let $f:B\to \Rn$ be $K$-quasiregular satisfying the growth condition~\eqref{cond-g} with some nonnegative exponent and  the multiplicity condition~\eqref{precond-m} with $0\leq a< n-1$. Then $f$ has non-tangential limits at all points in $\Sn$ except possibly of the set $E$ of the Hausdorff dimension $\dim_{H}(E)<\frac{na}{1+a}$.
\end{theorem}

As an application of the above result we show the existence and integrability of the boundary limit functions of quasiregular mappings in Hardy spaces.

\begin{cor}\label{cor-rad-lim}
   Let $f:B\to \Rn$ be $K$-quasiregular mapping in Hardy space $\Hp$. Additionally, if $f$ satisfies the multiplicity condition~\eqref{cond-m} with $0\leq a<n-1$, then the non-tangential limit function $\tilde{f}$ exists a.e. in $\Sn$ and $\tilde{f}\in L^p(\Sn)$. 
%\komT{In the previous version we had here assumption that $a<1$, but $a<n-1$ works as well.}
\end{cor}

\begin{proof}
 By assumptions and by Observation~\ref{thm1} we have that $f$ satisfies both the growth and multiplicity conditions  \eqref{cond-g} and \eqref{cond-m}, respectively. Therefore, Theorem~\ref{thm-kmv} immediately implies the first assertion of the corollary. In particular, since $a<n-1$ the non-tangential limit function $\tilde{f}$ is defined except the set $E$ of the Hausdorff dimension $\dim_H(E) < \frac{na}{a+1}<n-1$. The $L^p$-integrability of $\tilde{f}$ follows from Proposition~\ref{prop1}, as the value of $\tilde{f}$ is independent of choice of a curve non-tangentially approaching the given boundary point, and thus we can reduce our investigations to the radial limits only.
\end{proof}
%\komT{
%The above discussion shows that the class of quasiregular mappings in the Hardy spaces can be (trivially) characterized as follows.
%
%\begin{cor}
%  Let $f:B\to \Rn$ be $K$-quasiregular  satisfying the multiplicity condition~\eqref{cond-m} and the growth condition~\eqref{cond-g}. Then $f\in \Hp$ if and only if the radial limits function $\tilde{f}\in L^p(\Sn)$.
%\end{cor}
%\begin{proof}
% Let $f\in \Hp$. Then, by Theorem 4.1 and Remark 4.2 (1) in~\cite{kmv}, we know that $\tilde{f}$ exists and by the above Proposition~\ref{prop1} belongs to $L^p(\Sn)$. The opposite implication follows from the Lebesgue dominated convergence theorem, by $f(x)\leq \tilde{f}(\om)$ for $x=r\om$ for $0<r<1$, and by the definition of $\Hp$.
%\end{proof}
%}

%\komT{
% Suppose that $f$ is $QR$ and $f\in \Hp$. If we know that the non-tangential maximal function $f^{*}\in L^p(\Sn)$, then by modifying the proofs in~\cite{AK} we may prove the following results:
%
% - The sufficiency part of Lemma 4.5 in~\cite{AK} (and hence the full characterization of Carleson's measures).
% 
% - the implication (3.4) to (3.5) in Theorem 3.3~\cite{AK} (by Lemma 4.5).
%
%}
We recall the following estimate by Miniowitz. It holds for quasiregular mappings omitting a set $E \subset \mathbb{R}^n$ such that
\begin{equation}\label{cond:omit-set}
E\cap S^{n-1}(r)\not=\emptyset,\quad \hbox{for all } r\geq 0.
\end{equation}
For example, one can think about the generic $E$ as the nonnegative part of the $x_n$-axis: 
\[
E:=\{(x_1,\ldots,x_n)\in \Rn: x_1=\ldots=x_{n-1}=0, x_n\geq 0\}.
\]
\begin{theorem}[cf. Theorem 3 in~\cite{M1}]\label{thm-min}
 Let $f:B\to \Rn\setminus E$ be a $K$-quasiregular mapping for set $E$ as in~\eqref{cond:omit-set}. Then 
 \begin{equation}\label{est-min}\tag{Min}
\frac{1}{C}\left(\frac{1-r}{1+r}\right)^{\alpha} \leq \frac{|f(y)|}{|f(0)|} \leq C\left(\frac{1+r}{1-r}\right)^{\alpha},\quad r=|y|,
 \end{equation}
 where $\alpha=2^{n-1}K_I(f)$ and $C=2^{8\alpha}$.
\end{theorem}
Recall that $K_I(f)$ stands for the inner distortion of $f$ and that for $K$-quasiregular mappings $K_I\leq K$, see \cite[Chapter I.2]{Ric}.

We remark, that for $n=2$ the constant $C=4^{\alpha}$, see~\cite[Formula (6.2)]{M1}. Moreover, Theorem~\ref{thm-min} has counterparts for quasiregular mappings with bounded multiplicity and spherically mean $1$-valent quasiregular mappings, see~\cite{M1}. However, for quasiregular local homeomorphisms the corresponding estimate has different forms: the exponent in the growth condition depends on $r$ (cf.~\cite[Theorem 2]{M1}) or the growth estimate is of exponential type (cf. Theorem in Remarks on pg. 69 in~\cite{M1}).

The conditions on the size of the omitted set $E$ in the assumptions of Theorem~\ref{thm-min} depend largely on the imposed assumptions on the multiplicity of the underlying quasiregular map.  Indeed, for quasiregular mappings with bounded multiplicity N, estimate~\eqref{est-min} is valid when $f$ omits merely the origin. In this case the constants C and $\alpha$ also depend on N, see Theorem 1  in~\cite{M1}. Therefore, it would be interesting to construct examples of quasiregular maps satisfying our multiplicity condition~\eqref{cond-m}, estimates~\eqref{est-min} which however need not to satisfy condition~\eqref{cond:omit-set} and omit e.g. only the origin.
\smallskip

	The following auxiliary observation turns out to be crucial in a number of results below, including Lemma~\ref{lem-Jones-qr}. 
	
	\emph{Fix  $x\in B$ and let $T_x$ be  the conformal selfmapping as in~\eqref{def:Tx}. If a quasiregular mapping $f$ omits a set $E$, then so does $g_x:=f\circ T_x^{-1}$, for any $x\in B$. Therefore, $g_x$ satisfies ~\eqref{est-min} with the same constants as $f$. Moreover, $g_x$ satisfies the multiplicity condition~\eqref{precond-m}, i.e. the multiplicity condition~\eqref{cond-m} only on hyperbolic balls centered at zero.}
\smallskip

 The above discussion, in particular the variety of mappings satisfying the assertion of Theorem~\ref{thm-min} and the generality of possible omitted sets in the assumption of this theorem, allow us to define the following class of mappings:
\smallskip

\noindent
{\bf The Miniowitz class}. \emph{We denote by  ${\rm (Min)}$ the class of all $K$-quasiregular mappings  $f: B \to \Rn\setminus\{0\} $ such that for any $x\in B$,  $g_x:=f\circ T_x^{-1}$ satisfies the growth estimate~\eqref{est-min} with some positive constants $C$ and $\alpha$ depending only on $n$ and $K$.}
\smallskip

We remark that a map in class ${\rm (Min)}$ satisfies the growth condition~\eqref{cond-g} with constant $2^\alpha C |f(0)|$ and $\beta:=\alpha$ for $\alpha$ and $C$ as in Theorem~\eqref{thm-min}.

As a consequence of Theorem~\ref{thm-min} and the above properties of $g_x$ we obtain the Harnack inequality on hyperbolic balls contained in $B$. Once we show the Jones lemma for quasiregular mappings (Lemma~\ref{lem-Jones-qr}), we present yet another proof of the Harnack inequality based on the potential theoretic properties of quasiregular mappings, see Remark~\ref{rem-harnack}.

\begin{cor}[Harnack inequality for quasiregular mappings]\label{cor:harnack}
 Let $f:B\to \Rn \setminus \{0\}$ be a $K$-quasiregular mapping in class~\eqref{est-min}. Then for all  $x\in B$ and all $y, z \in B(x, k(1-|x|))$, for any given $0<k<\frac12 (\sqrt{5}-1)$, it holds that
 \begin{equation}\label{harnack-balls}
  \frac{|f(y)|}{|f(z)|} \approx_{_{n, K, k}} 1.
 \end{equation}
\end{cor}

\begin{proof}
%	\textcolor{blue}{I think that we should ask $ k<1/4$, then set $r_0=k^2+2k<3/4$. Let $y'= T_x(y)$ and $z'=T_x(z)$.Then, both $y'$ and $z'$ are contained in $B(0,3/4)$. By the definition of (Min) class
%	$$
%\frac{|f(y)|}{|f(z)|}= \frac{|f(y)|/|f(x)|}{|f(z)|/|f(x)|}= \frac{|g_x(y')|/|g_x(0)|}{|g_x(z')|/|g_x(0)|}.
%	$$
%	Then estimate \ref{harnack-balls} follows immediately from the definition of the class (Min)}
Let $y\in B(x, k(1-|x|))$ for some $x\in B$. Then, it holds that
\[
 (1-k)(1-|x|)\leq 1-|y|\leq (1+k)(1-|x|).
\]
Hence, estimate~\eqref{est-min} applied to points $y,z \in B(x, k(1-|x|))$ results in the following:
 \[
 \frac{|f(y)|}{|f(z)|} \leq C^2 \left(\frac{1+k(1-k)(1-|x|)}{1-k(k+1)(1-|x|)}\right)^{2\alpha}\leq \frac{2^{\alpha}C^2}{(1-k(k+1)(1-|x|))^{2\alpha}}\leq \frac{2^{\alpha}C^2}{(1-k(k+1))^{2\alpha}}.
 \]
Similarly, we obtain the lower estimate.
% First, let us consider the case of a hyperbolic ball $B(0,k)$ centered at the origin $x=0$ with radius $0<k<\sqrt{2}-1$. Then, for any $y\in B(0, k)$, the estimate~\eqref{est-min} results in the following:
% \[
% \frac{|f(y)|}{|f(0)|} \approx C \left(\frac{1+k}{1-k}\right)^{\alpha},
% \]
%By applying these estimates at points $y, z \in B(0, k)$ we obtain assertion~\eqref{harnack-balls} for balls centered at $0$. Next consider the general case of $y, z\in B(x, k(1-|x|))$ for some $x\in B$. Then, by assertion (3) of Lemma~\ref{lem:Tx}, it holds that $f(y)=f(T_x^{-1}(y'))$ for some $y'\in B(0, k^2+2k)$ and similarly for point $z$. Therefore, since by the second assertion of Proposition~\ref{prop-multi} the map $g=f\circ T_x^{-1}$ satisfies \eqref{est-min}, we conclude that
% \[
%   \frac{|f(y)|}{|f(z)|}= \frac{|f(T_x^{-1}(y'))|}{|f(T_x^{-1}(z'))|}= \frac{|g(y')|}{|g(z')|} \approx \frac{2^{\alpha}C^2}{(1-|x|)^{4\alpha}} \frac{1}{(1-k^2-2k)^{4\alpha}}.
% \]
\end{proof}

We are in a position to prove one of the most important observations of this work, namely that a quasiregular map defines a class of the Carleson measures. Similar results are due to Jones for planar quasiconformal mappings, see~\cite[Lemma 4.2]{j}, and by Astala--Koskela for quasiconformal mappings in $\R^n$, see~\cite[Lemma 5.6]{AK} and also~\cite[Proposition 4.25]{af} for a variant of the following lemma for quasiconformal mappings in the Heisenberg group.

Recall that a positive Borel measure $\mu$ on $B$ is a \emph{Carleson measure} if there exists a constant $C>0$ such that
\begin{equation*}
\mu(B \cap B(\omega,r))\leq C r^{n-1},\quad \text{for
all }\omega\in \partial\Omega\text{ and }r>0.
\end{equation*}
The Carleson measure constant of $\mu$ is defined as the infimum of all constants $C$ such that the above inequality holds. In Section 4 we discuss also more general Carleson measures, see Definition~\ref{def: Carleson-msp}.

\begin{lem}[Jones lemma for quasiregular mappings]\label{lem-Jones-qr}
  Let $n\geq 2$ and $0<q<n$. Moreover, let $f:B\to \Rn \setminus\{0\}$ be a $K$-quasiregular mapping in class~\eqref{est-min} satisfying multiplicity condition~\eqref{cond-m} with $0 \leq a<n-1$. Then the following measure $\mu$ is the Carleson measure in $B$:
 \[
  \ud\mu=\frac{|Df(x)|^q}{|f(x)|^q}(1-|x|)^{q-1} \ud x. 
 \]
 Furthermore, the Carleson constant of $\mu$ depends only on $n, q$ and $K, \alpha, a$.
\end{lem}
Before proving the lemma we make some remarks.

\begin{remark}
\
\begin{itemize}
\item[(1)] We would like to point to and emphasize the robustness of the Jones lemma and some further possible developments concerning the lemma. One of the key tools employed below in the proof of that lemma is the Miniowitz estimate~\eqref{est-min}. As mentioned in the discussion following the statement of Theorem~\ref{thm-min}, its assertion has numerous variants for quasiregular mappings under various multiplicity and local injectivity conditions. The similar observation applies to the Jones lemma. Namely, by repeating the first part of the proof below, i.e. formulas \eqref{int:l5.6}-\eqref{est-Carleson-global}, one can obtain the global estimate for the integral:
\begin{equation*}
 \int_{B}\frac{|Df(x)|^q}{|f(x)|^q}(1-|x|))^{q-1} \ud x <\infty. %\leq c(a,C, K, \alpha, n, q),
\end{equation*}
for $0<q<n$ and local homeomorphic quasiregular maps (cf. \cite[Theorem 2]{M1}), as well as quasiregular maps omitting the origin with appropriate growth of the multiplicity (cf. \cite[Lemma 2]{M1}). In both cases the Miniowitz estimates are technically involved, as the exponent $\alpha$ may vary from point to point in ball $B$.
Such a global integrability property plays a role in the studies of the non-tangential limits of the quasiregular mapping, see the proof of Theorem 4.1 in~\cite{kmv} and Theorem 3.1 therein. However, the second part of the proof of Lemma~\ref{lem-Jones-qr} for a composition of a qr map with the M\"obius transformation $T_x$ leads for the local homeomorphic qr maps to the tedious and long computations. Therefore, we decided to present the above version of the Jones lemma only. Nevertheless, it is worthy to pursue the studies of the Carleson measures in the setting of locally homeomorphic qr maps, as having a counterpart of Lemma~\ref{lem-Jones-qr} could give the characterization of the Hardy spaces for that class of mappings, cf.~Theorem~\ref{thm-main}. 

\item[(2)]
In the second part of proof below we employ assertions (1) and (2) in Lemma~\ref{lem:Tx}. There we have a technical restriction on the radii of the Carleson balls (i.e. $r<\frac12$). However, this affects the Carleson condition only in terms of the constant. Indeed, let $r\geq \frac12$, then by~\eqref{est-Carleson-global} we have  that $\mu$ is a Carleson measure on ball $B$, and so
\begin{equation*}
\mu(B\cap B(\om, r))\leq \mu(B) \leq c(a,\alpha, n, p, K) \approx_{c} \left(\frac{r}{r}\right)^{n-1} 
  \lesssim_{c} 2^{n-1} r^{n-1}.
\end{equation*}
\end{itemize}
\end{remark}

\begin{proof}
 We follow the general framework of the proof of Lemma 5.6 in~\cite{AK}. However, the lack of injectivity of a map requires modifying the approach in~\cite{AK}. First, we show that the Carleson condition holds on entire ball, namely that $\mu(B)<c(a, \alpha, n, q, K)$, where $\alpha$ is the exponent in Theorem~\ref{thm-min} and $a$ stands for the exponent in condition~\eqref{cond-m}.
 
 Let us define the ring domains $R_j=\{1-2^{-j}\leq |x| \leq 1-2^{-(j+1)}\}$ for $j=0,1,\ldots$. Let $f$ be any quasiregular map satisfying the assumptions of the lemma. Then for $\ep>0$ to be determined later, by applying H\"older's inequality, we obtain that
 \begin{align}
 \int_{B}\frac{|Df(x)|^q}{|f(x)|^q}(1-|x|)^{q-1} \ud x &=  \sum_{j=0}^{\infty}\int_{R_j}\frac{|Df(x)|^q}{|f(x)|^q}(1-|x|)^{q-1} \ud x \nonumber \\
 &\leq \sum_{j=0}^{\infty} \left(\int_{R_j}\frac{|Df(x)|^n}{|f(x)|^n}(1-|x|)^{\ep\frac{n}{q}} \ud x \right)^{\frac{q}{n}} \left(\int_{R_j} (1-|x|)^{\frac{n(q-1-\ep)}{n-q}} \ud x\right)^{\frac{n-q}{n}}. \label{int:l5.6}
 \end{align} 
 In order to estimate first integral in~\eqref{int:l5.6} we apply the quasiregular change of variables and the multiplicity condition~\eqref{cond-m} on $R_j\subset B(0, 1-2^{-j-1})$ which result in the following estimate
 \begin{align}
 \int_{R_j}\frac{|Df(x)|^n}{|f(x)|^n}(1-|x|))^{\ep\frac{n}{q}} \ud x  &\leq K \int_{R_j}\frac{J_f(x)}{|f(x)|^n}(1-|x|)^{\ep\frac{n}{q}} \ud x \nonumber \\ 
 & \lesssim_{C, K} 2^{a(j+1)}\int_{f(R_j)} \frac{(1-|f^{-1}(y)|)^{\ep\frac{n}{q}}}{|y|^n}\,\ud y.\label{lem-Jones-int1}
 \end{align}
Next, we split the estimates into two integrals, one over the set $f(R_j)\cap B(0, |f(0)|)$ and the other $f(R_j)\setminus B(0, |f(0)|)$. By applying the polar coordinates and Theorem~\ref{thm-min} we get
\begin{align*}
\int_{f(R_j)\cap B(0, |f(0)|)} \frac{(1-|f^{-1}(y)|)^{\ep\frac{n}{q}}}{|y|^n}\,\ud y
&\lesssim_{C, K, \alpha} \int_{B(0, |f(0)|)} |y|^{-n}\left(\frac{|y|}{|f(0)|}\right)^{\frac{\ep n}{\alpha q}}\,\ud y\\
&\approx_{C, K, \alpha} \int_{B(0, 1)} |z|^{-n} |z|^{\frac{\ep n}{\alpha q}}\,\ud z \qquad {\scriptstyle{\left(z:=\frac{y}{|f(0)|}\right)}}\\
&\approx_{C, K, \alpha} \omega_n\,\left[\frac{\alpha q}{\ep n} t^{\frac{\ep n}{\alpha q}}\right]_{0}^{1}. \\
&\approx_{\omega_n, C, K, \alpha} \frac{\alpha q}{\ep n}.
\end{align*} 

The integral converges if $\ep > 0$. In analogous way we estimate the second integral, since it holds that 
\[
f(R_j)\setminus B(0, |f(0)|) \subset \{|y|> |f(0)|\}.
\]
Hence, we have
\begin{align*}
\int_{f(R_j)\setminus B(0, |f(0)|)} \frac{(1-|f^{-1}(y)|)^{\ep\frac{n}{q}}}{|y|^n}\,\ud y
&\lesssim_{C, K, \alpha} \int_{\{|y|> |f(0)|\}} |y|^{-n}\left(\frac{|f(0)|}{|y|}\right)^{\frac{\ep n}{\alpha q}}\,\ud y \\
&\approx_{C, K, \alpha} \int_{\{|z|> 1\}} |z|^{-n} |z|^{-\frac{\ep n}{\alpha q}}\,\ud z\qquad {\scriptstyle{\left(z:=\frac{y}{|f(0)|}\right)}}\\
&\approx_{C, K, \alpha} \om_n\,\left[t^{-\frac{\ep n}{\alpha q}}\right]_{1}^{+\infty}\\
&\approx_{\om_n, C, K, \alpha} \frac{\alpha q}{\ep n}.
\end{align*} 
As previously, the integral converges provided that $\ep>0$. By applying the above two estimates in~\eqref{lem-Jones-int1}, we obtain for $j=0,1,\ldots$ that
\begin{equation}\label{est-Car-1}
 \left(\int_{R_j}\frac{|Df(x)|^n}{|f(x)|^n}(1-|x|))^{\ep\frac{n}{q}}\right)^{\frac{q}{n}} \ud x \leq 2^{a(j+1)\frac{q}{n}} c(C, K, \alpha, q, n, \ep).
\end{equation}
Regarding the second integral in~\eqref{int:l5.6}, we have that 
$$
|R_j|\approx (1-2^{-(j+1)})^n-(1-2^{-j})^n\leq n(1-2^{-j-1})^{n-1}(2^{-j}-2^{-(j+1)}) \lesssim \frac{n}{2^{j+1}}.
$$
Since $2^{-j-1} \leq 1-|x|\leq 2^{-j}$, we have that $1-|x|\approx 2^{-j}$ and so the following estimate holds
\begin{equation}\label{est-1}
\left(\int_{R_j} (1-|x|)^{\frac{n(q-1-\ep)}{n-q}} \ud x\right)^{\frac{n-q}{n}}\lesssim \left( \frac{n}{2^{j+1}}\, 2^{-j(q-1-\ep)\frac{n}{n-q}} \right)^{\frac{n-q}{n}} \lesssim 2^{\frac{q-n}{n}} 2^{-j(q-\frac{q}{n}-\ep)}.
\end{equation}
%Second, if $q-1-\ep<0$ then we estimate the integrand by $2^{-j-1}$ and get similar estimate
%\begin{equation}\label{est-2}
%\left(\int_{R_j} (1-|x|)^{\frac{n(q-1-\ep)}{n-q}} \ud x\right)^{\frac{n-q}{n}}\lesssim \left(n 2^{-j-1}\, 2^{-(j+1)(q-1-\ep)\frac{n}{n-q}} \right)^{\frac{n-q}{n}} \lesssim 2^{\frac{q-n}{n}-q+1+\ep} 2^{-j(q-\frac{q}{n}-\ep)}.
%\end{equation}
We substitute the above computations in~\eqref{int:l5.6} and, therefore, obtain the following estimate:
\begin{equation}\label{est-Carleson-global-pre}
 \int_{B}\frac{|Df(x)|^q}{|f(x)|^q}(1-|x|))^{q-1} \ud x \lesssim \sum_{i=0}^{\infty} 2^{a(j+1)\frac{q}{n}}2^{-j(q-\frac{q}{n}-\ep)}\lesssim \sum_{i=0}^{\infty} 2^{j(a\frac{q}{n} -q+\frac{q}{n}+\ep)} <\infty,
\end{equation}
provided that $a\frac{q}{n} -q+\frac{q}{n}+\ep<0$. This holds, if
$$
0<\ep<q\left(1-\frac{a+1}{n}\right).
$$ 
%and the necessary condition $\ep>0$ imposes additionally that $a<n-1$. If $q=1$, then in order to establish estimate \eqref{est-Carleson-global} we need to follow \eqref{est-2} (which then reads $\ep>0$). Otherwise, if $q\not=1$, then we follow \eqref{est-1} and get that
%\[
% \ep < \min\{q(1-\frac{a+1}{n}), q-1\}, 
%\]
%from which the necessary condition $q>1$, as $\ep>0$. In both cases there is a room to choose $\ep>0$ in all dimensions.
 
 Therefore, we have proven the global integrability of the following expression:
\begin{equation}\label{est-Carleson-global}
 \int_{B}\frac{|Df(x)|^q}{|f(x)|^q}(1-|x|))^{q-1} \ud x \leq c(a,C, K, \alpha, n, q).
\end{equation}

We will now show how this observation implies the assertion of the theorem. Let $\om\in \partial B$ and consider the Carleson tent $B\cap B(\om,r)$ for some $0<r<\frac12$. Furthermore, choose $x_0\in B\cap B(\om,r)$ such that $x_0$ lies on the radial segment joining the origin with $\om$ and $|x_0-\om|=\frac{r}{2}$, cf. the proof of Lemma~\ref{lem:Tx}.

% Recall, that conformal map $T_{x_0}$ has the following properties:
%\begin{align}
% &J_{T_{x_0}}\approx \frac{1}{r^n}\quad\hbox{ on } B\cap B(\om,r),\\
% &J_{T_{x_0}}(x)\approx \frac{1-|T_{x_0}(x)|}{1-|x|} \approx \frac{1}{r}\quad\hbox{ for } x\in B\cap B(\om,r).
%\end{align}
Hence, by Lemma~\ref{lem:Tx} properties (1) and (2) of $T_{x_0}$, it holds that
\begin{align}
 |DT_{x_0}(x)|^{q}&=J_{T_{x_0}}(x)^{\frac{q}{n}}=\frac{J_{T_{x_0}}(x)^{\frac{q}{n}}}{J_{T_{x_0}}(x)}J_{T_{x_0}}(x) \nonumber \\
 &\approx r^{n-1}J_{T_{x_0}}(x)\frac{1}{r^{q-1}}\approx r^{n-1}J_{T_{x_0}}(x) \frac{(1-|T_{x_0}(x)|)^{q-1}}{(1-|x|)^{q-1}}. \label{est-T-prop}
\end{align}
Set 
$$
g:=g_{x_0}:=f\circ T_{x_0}^{-1}.
$$
Since $T_{x_0}$ is conformal and $f$ is quasiregular, so is map $g$. It holds by~\eqref{est-T-prop} that
\begin{align}
 &\int_{B\cap B(\om,r)}\frac{|Df(x)|^q}{|f(x)|^q}(1-|x|)^{q-1} \ud x \leq \int_{B\cap B(\om,r)}\frac{|Dg(T_{x_0})(x)|^q}{|g(T_{x_0}(x))|^q} |DT_{x_0}(x)|^{q}\, (1-|x|)^{q-1}\ud x \nonumber \\
 &\approx r^{n-1} \int_{B\cap B(\om,r)}\frac{|Dg(T_{x_0})(x)|^q}{|g(T_{x_0}(x))|^q} J_{T_{x_0}}(x) (1-|T_{x_0}(x)|)^{q-1}\,\ud x \nonumber\\
 &= r^{n-1}\int_{T_{x_0}(B\cap B(\om,r))}\frac{|Dg(y)|^q}{|g(y)|^q}\,(1-|y|)^{q-1} \ud y \nonumber \\
 &\leq r^{n-1}\int_{B}\frac{|Dg(y)|^q}{|g(y)|^q}\,(1-|y|)^{q-1} \ud y \lesssim r^{n-1}. \label{est-g-f}
\end{align}
In order to explain the last estimate notice that, by the discussion before the definition of the Miniowitz class, $g$ satisfies~\eqref{est-min}. Therefore, upon repeating estimates~\eqref{int:l5.6} - \eqref{est-Carleson-global-pre} for $g$, we obtain~\eqref{est-Carleson-global}. This completes the proof of~\eqref{est-g-f} and, hence, the whole proof of the lemma is completed as well.
\end{proof}

\begin{remark}\label{rem-harnack}
 In Corollary~\ref{cor:harnack} above we show the Harnack inequality for quasiregular mappings satisfying the Miniowitz condition~\eqref{est-min}. However, one can provide another proof relying on Lemma~\ref{lem-Jones-qr} and on the potential theory for quasiregular mappings. Such a result is proven along the lines of the proof of~\cite[Proposition 2]{z},  under additional assumption that $|\nabla \ln |f||\ud x$ is a Carleson measure.

\emph{
Let $f:B\to \Rn$ be a $K$-quasiregular mapping such that $|\nabla \ln|f||\ud x$ is a Carleson measure. Then for all  $x\in B$ and all $y, z \in B(x, k(1-|x|))$, for any given $0<k\leq \frac12$, it holds that
 \begin{equation*}%\label{harnack-balls}
  \frac{|f(y)|}{|f(z)|} \leq C(n,p,K,k).
 \end{equation*}
%\end{cor}
}
In order to prove the estimate, set $u:=\ln |f|$. Then by the discussion following \cite[Theorem 14.42]{hkm}, \cite[Lemma 14.38]{hkm} and by~\cite[Corollary 6.2]{bi}, we have that $u\in W^{1,p}_{loc}(B)$ for some $p>n$. Then, by the Morrey embedding theorem, the Gehring lemma, see e.g.~\cite[(5.10)]{in}, and by~\cite[Theorem 2]{in} applied to function $|Df|$, we have that for all $y,z\in B(x, k(1-|x|))$ it holds:
\begin{align*}
 |u(y)-u(z)|&\lesssim_{p,n, K} |y-z|^{1-\frac{n}{p}}(k(1-|x|))^{\frac{n}{p}-n}\|\nabla u\|_{L^1\left(B(x, \frac32 k(1-|x|))\right)} \\
 &\lesssim_{p,n, K} k^{1-n}\frac{1}{(1-|x|)^{n-1}}\int_{B(x, \frac32 k(1-|x|))} |\nabla u| \leq C(n,p,K,k),
\end{align*}
where in the last inequality we employ that $|\nabla u|$ defines a Carleson measure.
\end{remark}

Before we state another consequence of the quasiregular Jones lemma, let us discuss the relations between the approach to the Hardy spaces based on the modulus of curve families, as in Section 4 in~\cite{AK} and our approach relying on the multiplicity and growth conditions imposed on a quasiregular mappings.  Some key results in~\cite{AK}, such as Theorems 3.3 and 4.1, Lemma 4.2 and Corollary 4.3 are based on the modulus definition of the quasiconformal mappings, see Section 2 in~\cite{AK} or some classical works \cite{Ric, va}. The definition of the $n$-modulus of curve family $M$ allows us to characterize nonconstant quasiregular mappings on ball $B$ as such sense-preserving, discrete and open maps that satisfy:
\[
 M(\Gamma)\leq K_0(f) N(f, A) M(f\Gamma),
\]
where $\Gamma$ is a family of paths in a Borel set $A\subset B$ such that the multiplicity $N(f,A)<\infty$, see Theorems II.2.4 and II.Theorem 6.7 in~\cite{Ric}. In the special case of quasiconformal mappings the multiplicity is one and we obtain the assertion of~\cite[Corollary II.2.7]{Ric}. In the setting of our work the lack of uniform bound for $N(f, A)$ is substituted by condition~\eqref{cond-m}, which however turns out to be too weak to repeat the core of reasonings in~\cite{AK}. Nevertheless, we still obtain the following quasiregular counterpart of the key observation~\cite[Corollary 4.3]{AK} that gives the fundamental characterization of the Hardy spaces for quasiconformal mappings in~\cite[Theorem 4.1]{AK}, see Remark~\ref{rem-Jones-main} following Theorem~\ref{thm-main} below.

First, we need to recall the following definition, here stated in the general metric space, as such generality is needed in Section 4 below. Let $(X,d)$ be a metric space and $\Omega \subset X$ a domain. We define the \emph{non-tangential region in $\Omega$ with parameter $\kappa$ centered at $\omega$} as follows:
\begin{equation}
\Gamma_{\Omega,\kappa}(\omega):=\{x \in \Omega:\; d(x,\omega) <
(1+\kappa)d(x,\partial \Omega)\}.\label{def-cone}
\end{equation}
When $\kappa$ is fixed, we write for simplicity $\Gamma_{\Omega}:=\Gamma_{\Omega,\kappa}$. 

Let now $X=\Rn$, $\Om=B$ and $d$ stand for the Euclidean distance. For a point $x\in B$ we define \emph{the shadow associated to $x$} as the following set (cf.~\cite[Definition 2.17]{af}):
$$
S(x):=\{\om\in \Sn: x\in \Gamma_B(\om)\}.
$$ 

\begin{prop}\label{prop-cor41}
  Let $n \geq 2$ and $f:B\to \Rn \setminus\{0\}$ be a $K$-quasiregular mapping in class~\eqref{est-min} satisfying the multiplicity condition~\eqref{cond-m} with $0\leq a<n-1$. Then for every $q>0$ and for all $x\in B$ it holds that
   \[
   |f(x)|^q\leq C \vint_{S(x)} |\tilde{f}(\om)|^q \ud \sigma,
   \]
  where $\ud \sigma$ stands for the surface measure on $\Sn$ and $C=C(n,q, K, a, \alpha)$.
\end{prop}

\begin{proof} First, let us observe that, since $f$ satisfies~\eqref{est-min}, then in particular it has the growth condition~\eqref{cond-g} with $\beta=\alpha$ and the constant $2^\alpha C |f(0)|$, where $C$ as in Theorem~\ref{thm-min}. This together with~\eqref{cond-m} allows us to infer from Theorem~\ref{thm-kmv} the existence of the non-tangential limit function $\tilde{f}$ a.e. 

We first prove two observations which are the quasiregular counterparts of Lemmas 4 and 5 in~\cite{z}. Since their proofs follow strictly the lines of~\cite{z}, we will restrict only to presenting the key steps of reasonings.

Let $u:=\ln |f|$ and $F(\om):=\int_{0}^{1}|\nabla u(t\om)|\,t^{n-1}\ud t$ for $\om\in \Sn$. Note that Lemma~\ref{lem-Jones-qr} implies the following
\begin{equation}
 \|F\|_{L^1(\Sn)}=\int_{\Sn} \int_{0}^{1}|\nabla u(t\om)|\,t^{n-1}\ud t \ud \om \leq \int_{B} \frac{|Df|}{|f|} \ud x<c(a, \alpha, n, K). \label{eq-proof-Prop29}
\end{equation}

\noindent {\sc Claim 1 }(cf. \cite[Lemma 4]{z}): \emph{If $N>C(K, \alpha)^2$ for constant $C(K, \alpha)$ in the Harnack inequality~\eqref{harnack-balls}, then it holds that
\begin{equation}\label{eq: lem4-Z}
 \sigma\left(\left\{\om\in \Sn: |f(x)|\leq \frac{|f(0)|}{N} \right\}\right)\leq \frac{7^n \|F\|_{L^1(\Sn)}}{\ln N}.
\end{equation}
}
\emph{Proof:} Denote the set of points in the assertion of the claim by $F_N\subset \Sn$. Then by the direct computations and Corollary~\ref{cor:harnack} applied for $x=0$ and $k=\frac17$, we have that for $\om\in F_N$
\[
 F(\om)\geq 7^{1-n}\int_{\frac17}^{1} |\nabla u(t\om)|\ud t\geq 7^{1-n}\left|\ln \frac{|f(\om)|}{f(\frac17 \om)}\right|
 \geq 7^{1-n}\left|\ln \frac{N}{C(K, \alpha)}\right|\geq \frac{\ln N}{7^n}.
\]
 The claim now follows, as $\sigma (F_N)\leq \frac{7^n \|F\|_{L^1(\Sn)}}{\ln N}$.
\smallskip
 
\noindent {\sc Claim 2 }(cf. \cite[Lemma 5]{z}): \emph{There exists a constant $C$ independent of $f$ such that the following estimate holds for all $z\in B$
\begin{equation}\label{eq: lem5-Z}
 \sigma\left(\left\{x\in S(z): |f(x)|\leq \frac{|f(z)|}{N} \right\}\right)\leq \frac{7^n C \|F\|_{L^1(\Sn)}}{\ln N} \sigma(S(z)).
\end{equation}
}
\emph{Proof:}  Set $g:=f\circ T_z^{-1}(x)$ and recall that $g$ satisfies the multiplicity condition~\eqref{cond-m} on hyperbolic balls centered at the origin. Since by property (4) in~\cite{z} of map $T_z$  we have that for all $z\in B$
\[
 B(0,\frac17)\subset T_z\left( B(z, \frac14 (1-|z|))\right),
\]
we conclude by Corollary~\ref{cor:harnack} with $k=\frac14$, that $|g(0)|\leq C(f)|g(\frac17 \om)|$ for any $\om\in \Sn$.
Moreover, by analogy to Claim 1 above, we set $u:=\ln |g|$ and then the corresponding function $F$ satisfies
$\|F\|_{L^1(\Sn)}\leq \int_{B}\frac{|Dg|}{|g|}<\infty$, by the estimate analogous to \eqref{est-g-f} in the proof of Lemma~\ref{lem-Jones-qr}. Thus, since $g(0)=f\circ T_z^{-1}(0)=f(z)$, Claim 1 implies that
\[
\sigma\left(\left\{\om\in \Sn: |g(x)|\leq \frac{|f(z)|}{N} \right\}\right)\leq \frac{7^n \|F\|_{L^1(\Sn)}}{\ln N},
\]
and so
\[
\sigma\left(T_z\left(\left\{y \in S(z): |f(y)|\leq \frac{|f(z)|}{N} \right\}\right)\right)\leq \frac{7^n \|F\|_{L^1(\Sn)}}{\ln N}.
\]
Finally, recall the following estimate (3) in~\cite{z} for $x,y\in S(z)$, the shadow associated with $z$:
\[
\frac{1}{9(1-|z|)} |x-y|\leq |T_z(x)-T_z(y)|\leq \frac{2}{1-|z|} |x-y|.
\]
 This, together with the observation that $\sigma(S(z)) \approx (1-|z|)^{-1}$ gives the assertion of Claim 2.

We are in a position to complete the proof of the proposition. Let us choose $N_0$ large enough so that the constant in the estimate in Claim 2 satisfies $\frac{7^n C \|F\|_{L^1(\Sn)}}{\ln N_0}=\frac12$. Then the assertion of Claim 2 implies:
\[
 \sigma\left(\left\{x\in S(z): |f(x)|>\frac{|f(z)|}{N} \right\}\right)\geq \frac12 \sigma(S(z)).
\]
Hence the assertion of the proposition follows immediately, as
\[
 \int_{S(x)}|\tilde{f}(\om)|^q~ \ud \sigma (\om)\geq \frac{|f(x)|^q}{N_0^q}\sigma\left(\left \{y \in S(x); |f(y)|> \frac {|f(x)|}{N_0}\right\}\right)\geq \left(\frac{1}{2N_0^q}\right) |f(x)|^q \sigma(S(x)).
\]
\end{proof}

\section{Characterization of Hardy spaces for quasiregular mappings}\label{section-4}

In this section we study the equivalent characterizations of the Hardy spaces $\mathcal{H}^p$ and prove the main result of our work, i.e. Theorem~\ref{thm-main}. As a consequence, we show a counterpart of the Jerison-Weitsman theorem, providing conditions implying that a quasiregular map on $B$ belongs to $\mathcal{H}^p$, see Corollary~\ref{cor: jer-weits}. Moreover, we discuss a characterization of the Carleson measures on $B$ via the integral norms of non-tangential limit functions, see Proposition~\ref{obs:cor45}. Further applications of Theorem~\ref{thm-main} encompass an integral characterization of $\mathcal{H}^p$ spaces in Lemma~\ref{lem-48} and the discussion on Bergman spaces of quasiregular mappings, see Theorem~\ref{t: Ap-spaces}. Finally, in the last section we address how the results for quasiregular mappings in $\mathcal{H}^p$ can be expressed for the related $\A$-harmonic functions, in particular for $p$-harmonic functions in the plane.

 We begin with proving the main theorem of our work.

%\begin{theorem}\label{thm-main}
%  Let $n \geq 2$ and $f:B\to \Rn$ be a $K$-quasiregular mapping in class~\eqref{est-min} satisfying the multiplicity condition~\eqref{cond-m} with $0\leq a<n-1$. Then the following conditions are equivalent for every $p>0$:
%  \begin{itemize}
%   \item[(1)] $f\in \mathcal{H}^p$,
%   \item[(2)] the non-tangential limit function $\tilde{f}\in L^p(\mathbb{S}^{n-1})$,
%   \item[(3)] the non-tangential maximal function $f^*\in L^p(\mathbb{S}^{n-1})$.
%  \end{itemize}
%  Moreover, all the norms in conditions (1)-(3) are equivalent:
%    \[
%  \|f\|_{\mathcal{H}^p}\approx \|f^*\|_{L^p(\mathbb{S}^{n-1})} \approx \|\tilde{f}\|_{L^p(\mathbb{S}^{n-1})},
%  \]
%  where the equivalence constants depend only on $n, p $ and $a, K, \alpha$.
%\end{theorem}

First, we recall the following definition and the result.

If $f:B\to \Rn$ is any map, then the \emph{non-tangential maximal function of f} is defined as follows:
\[
 f^{*}(\om):= \sup_{y\in \Gamma_B(\om)} |f(y)|,\quad \om\in \Sn
\]
where $\Gamma_B(\om)$ stands for a non-tangential region in $B$ centered at $\om$, see~\eqref{def-cone}.

\begin{prop}[cf. Proposition 1 in~\cite{z}]\label{prop-zin}
 Let $f:B\to \Rn\setminus\{0\}$ be such that $f\in W^{1,1}_{loc}(B)$. Set $u:=\ln|f|$. If $|\nabla u|\ud x$ is a Carleson measure and $f$ satisfies the Harnack inequality, then for every $p>0$, it holds that
 \[
     \|f^*\|^p_{L^p(\mathbb{S}^{n-1})}\leq C \|\tilde{f}\|^p_{L^p(\mathbb{S}^{n-1})},
 \]
 where $C$ depends only on $n, p$ and $\|u\|_{*}$. 
 \end{prop}
 The norm $\|u\|_{*}$ appearing in the statement of \cite[Proposition 1]{z} for us corresponds to integral ~\eqref{est-Carleson-global} for $q=1$, see the discussion on pg. 127 in \cite{z}.

\begin{proof}[Proof of Theorem~\ref{thm-main}]
 
 {\bf (1) $\Rightarrow$ (2)} This follows immediately by Corollary~\ref{cor-rad-lim}, which gives the existence of the non-tangential function $\tilde{f}$ at a.e. point of $\mathbb{S}^{n-1}$.
%In particular, the value of $\tilde{f}$ is independent of choice of a curve non-tangentially approaching the given boundary point, and hence we can restrict the studies to the radial limits only. Then, Proposition~\ref{prop1} implies assertion (2).
% 

 {\bf (2) $\Rightarrow$ (3)} Let us apply Proposition~\ref{prop-zin} to map $f$ and $u=\ln |f|$. In order to verify assumptions of the proposition, we observe that, since $|\nabla u|\leq |Df|/|f|$, we have by Lemma~\ref{lem-Jones-qr} that $|\nabla u| \ud \mathcal{L}^n$ defines the Carleson measure. Moreover, by~\eqref{harnack-balls} the Harnack inequality holds for $f$.
%\komT{I will comment later that the dependence of constant in H.I. on $r$ in Cor. ~\ref{harnack-balls} is not a problem, as \cite{z} uses a fixed $r=\frac{6}{7}$.}\komT{I checked that \cite{z} actually uses $r=\frac17$, so with our assumption $0<r\leq \frac12$ we are perfectly fine.}
 Therefore, ~\cite[Proposition 1]{z} gives us that
\begin{equation}\label{est1:thm-main}
   \|f^*\|^p_{L^p(\mathbb{S}^{n-1})}\leq C   \|\tilde{f}\|^p_{L^p(\mathbb{S}^{n-1})},
\end{equation}
where $C$ depends on $n, p $ and $a, K, \alpha$. The assertion (3) follows.
 
 {\bf (3) $\Rightarrow$ (2)} This follows from $f(r\om)\leq f^{*}(\om)$ a.e. in $\mathbb{S}^{n-1}$ and small enough $r$. 
 
Finally, the same inequality allows us to infer that  {\bf (3) $\Rightarrow$ (1)}, by integration over $\mathbb{S}^{n-1}$ and using Definition~\ref{def-Hardy}. 
 \end{proof}

\begin{remark}\label{rem-Jones-main}
 Let us observe that under assumptions of Theorem~\ref{thm-main}, its assertion (2) together with Proposition~\ref{prop-cor41} imply assertion (3), similarly as in the proof of~\cite[Theorem 4.1]{AK}. Indeed, the estimate in Proposition~\ref{prop-cor41} implies that the non-tangential maximal function satisfies
 $$
 |f^*(\om)|^q\leq C M(|f|^q)(\om)\quad \hbox{for any } q>0,
$$
where $M$ stands for the Hardy--Littlewood maximal function on $\Sn$. Take $q$ such that $q<p$. Then by the boundedness of the operator $M$ on $L^s(\Sn)$ for $1<s<\infty$, we get the following:
 \begin{align*}
  \|f^{*}\|^p_{L^p(\Sn)}=  \|{|f^{*}|}^q\|^{\frac{p}{q}}_{L^{\frac{p}{q}}(\Sn)} \lesssim \|M(|f|^q)\|^{\frac{p}{q}}_{L^{\frac{p}{q}}(\Sn)}\lesssim \|\tilde{f}\|^p_{L^{p}(\Sn)}.
 \end{align*}
\end{remark}

\subsection*{The Jerison--Weitsman theorem for quasiregular maps}

One of the important consequences of Theorem~\ref{thm-main} and Lemma~\ref{lem-Jones-qr} is a counterpart of the Jerison--Weitsman result,~\cite[Theorem 1]{JW} showing that the poll of quasiregular mapping belonging to some Hardy space is wide. 
\begin{cor}\label{cor: jer-weits}
   Let $n\geq 2$ and $f:B\to \Rn \setminus\{0\}$ be a $K$-quasiregular mapping in class~\eqref{est-min} satisfying multiplicity condition~\eqref{cond-m} with $0\leq a<n-1$. Then, there exist constants $C=C(n, K, a, \alpha)$ and $p=p(n, K, a, \alpha)$ such that $ \|f\|_{H_{p}} \leq C$.
 \end{cor}
 \begin{proof}
  We recall that by the Varopoulos theorem (see e.g. \cite[Theorem 4]{z}) a Sobolev real-valued function $u\in W^{1,1}(B)$
  with radial limits function $\tilde{u}$, such that $|\nabla u| \ud x$ defines a Carleson measure satisfies $\tilde{u}\in {\rm BMO}(\mathbb{S}^{n-1})$. We apply this result to $u:=\ln |f|$ and appeal also to Lemma~\ref{lem-Jones-qr} with $q=1$ together with Corollary~\ref{cor-rad-lim}, to conclude that the radial limits function $\tilde{u}=\widetilde{\ln |f|}\in {\rm BMO}(\mathbb{S}^{n-1})$. Furthermore, the discussion similar to the one following Theorems 3 and 4 in~\cite{z} gives us that the above BMO-norm can be estimated by a constant $C(n, K, a, \alpha)$. For the readers convenience we present some more details. Define function $F$ as in~\eqref{eq-proof-Prop29}, corresponding to function $H_f$ in Zinsmeister's notation (cf. pg. 126 in~\cite{z}), i.e.:
 \[
  F(\om)=\int_{0}^{1}|\nabla u(t\om)|\,t^{n-1}\ud t \leq \int_{0}^{1}\frac{|Df(t\om)|}{|f(t\om)|}\,t^{n-1}\ud t:=H_f(\om).
 \] 
Then, by the discussion analogous to the one at formulas ~\eqref{est-T-prop} and ~\eqref{est-g-f}, we show by our Lemma~\ref{lem-Jones-qr} applied for $q=1$ that also $H_{f\circ T}\in L^1(\Sn)$ for any M\"obius selfmap $T$ of ball $B$. This gives the conclusion as in~\eqref{eq-proof-Prop29} and hence by the Varopoulos theorem~\cite[Theorem 4]{z} we have that
  \begin{equation}\label{est-BMO-cor29}
\| \tilde{u}\|_{{\rm BMO}(\mathbb{S}^{n-1})}=\|\widetilde{\ln |f|}\,\|_{{\rm BMO}(\mathbb{S}^{n-1})} \leq C(n, K, a, \alpha).
  \end{equation}
 Next, we apply the John--Nirenberg lemma, see e.g. Section 18 in~\cite{hkm}, to show that for some $p>0$ the $L^p$-norm of $\tilde{f}$ is finite.
% 
% get the uniform upper bound for the $L^1$-norms of functions $|f(r\om)|$ for almost every $0<r<1$, and hence that the non-tangential limit function $\tilde{f}\in L^1(\Sn)$. First, for any fixed ball $B_1:=B(\om, r_0)\Subset \Sn$ we consider instead of function $\tilde{u}$, a function $g:=\frac{\tilde{u}-\tilde{u}_{B_1}}{c(n) \|\tilde{u}\|_{{\rm BMO}(\mathbb{S}^{n-1})}}$, see pg. 339 in~\cite{hkm}. Then, we may assume that 
% $$
% g_{B_1}=0\quad\hbox{ and }\quad \|g\|_{{\rm BMO}(\mathbb{S}^{n-1})}=\frac{1}{c(n)}.
%$$
%Moreover, it is enough to consider only the unbounded case for $|f|$, and so we may also assume that $|f|>1$. Then, by the John-Nirenberg lemma we have that, for some exponent $p>0$ to be determined later, it holds
\begin{align*}
 \int_{B}|\tilde{f}|^p &\leq \int_{1}^{\infty}ps^{p-1} \sigma\left(\{x\in B: |\tilde{f}|\geq s\}\right)\ud s
 = \int_{1}^{\infty}ps^{p-1} \sigma\left(\{x\in B: \ln|\tilde{f}|\geq \ln s\}\right)\ud s \\
 &\lesssim \sigma(B)\,\int_{1}^{\infty}ps^{p-1}  e^{-c(n)\|\tilde{u}\|_{{\rm BMO}(\mathbb{S}^{n-1})}\ln s}\ud s 
 \lesssim p\sigma(B)\,\int_{1}^{\infty}s^{p-1-c(n)\|\tilde{u}\|_{{\rm BMO}(\mathbb{S}^{n-1})}}\ud s.
\end{align*}  
  The latter integral converges provided that $p<c(n)\|\tilde{u}\|_{{\rm BMO}(\mathbb{S}^{n-1})}$, which by~\eqref{est-BMO-cor29} gives that $p<C(n, K, a, \alpha)$. Finally, implication (2) to (1) in Theorem~\ref{thm-main} gives us the assertion of the corollary.
 \end{proof}

\subsection*{Characterization of Carleson measures via quasiregular mappings}\label{subs-Carl}

Analysis of implications (1) $\Rightarrow$ (2) and (2) $\Rightarrow$ (3) in Theorem~\ref{thm-main} together with a general result in metric measure spaces, see \cite[Theorem 6.13]{af} and below, allow us to infer an important characterization of the Carleson measures on ball $B$, previously known for quasiconformal mappings on balls in Euclidean spaces, see Proposition~\ref{obs:cor45} below (cf. also~\cite[Theorem 1.5]{af} for the similar result in the Heisenberg group $\mathbb{H}_1$).

Before presenting the result, let us recall the aforementioned metric measure spaces result and the necessary definitions.

A non-empty domain $\Om\subset X$ of a metric space $(X,d)$ has \emph{$s$-regular boundary} for some $s>0$, if its
boundary is Ahlfors $s$-regular with respect to the Hausdorff measure on $X$ restricted to $\partial \Om$, i.e., there exists a constant $C\geq 1$ such that
\begin{displaymath}
C^{-1}\, r^s \leq \mathcal{H}^s(B(x,r)\cap \partial \Omega)\leq C
\,r^s,\quad \text{for all }x\in \partial \Omega\text{ and
}0<r<\mathrm{diam}(\partial \Omega).
\end{displaymath}

Let us now recall a notion of Carleson measure in metric spaces.

\begin{defn}\label{def: Carleson-msp}
 Fix $1\leq \alpha<\infty$ and $s>0$. Let $(X,d)$ be a metric space and $\Omega \subset X$ a domain with nonempty $s$-regular boundary. We say that a (positive) Borel measure $\mu$ on $\Omega$ is an \emph{$\alpha$-Carleson measure on  $\Omega$} if there exists a constant $C>0$ such that
\begin{equation}\label{eq:CarlesonCOnstAbstr}
\mu(\Omega \cap B(\omega,r))\leq C r^{\alpha\,s},\quad \text{for
all }\omega\in \partial\Omega\text{ and }r>0.
\end{equation}
The $\alpha$-\emph{Carleson measure constant} of $\mu$ is defined by
\begin{displaymath}
\gamma_{\alpha}(\mu):= \inf\{C>0\text{ such that
\eqref{eq:CarlesonCOnstAbstr} holds for all $\omega \in\partial
\Omega$ and $r>0$}\}
\end{displaymath}
We also call $1$-Carleson measures simply \emph{Carleson
measures}.
\end{defn}
%We define the \emph{non-tangential region in $\Omega$ with parameter $\kappa$ centered at $\omega$} as follows:
%\begin{displaymath}
%\Gamma_{\Omega,\kappa}(\omega):=\{x \in \Omega:\; d(x,\omega) <
%(1+\kappa)d(x,\partial \Omega)\}.
%\end{displaymath}
%When $\kappa$ is fixed, we write for simplicity $\Gamma_{\Omega}:=\Gamma_{\Omega,\kappa}$.

\begin{theorem}\cite[Theorem 6.13]{af}\label{p:NontangentialInequality} Fix $s>0$ and $1\leq \alpha <\infty$. Let $(X,d)$ be a proper metric space and let $\Omega \subset X$ be a bounded
domain with nonempty $s$-regular boundary and let $\kappa>0$ be
such that the non-tangential region $\Gamma_{\kappa}(\omega)$ is
nonempty for all $\omega \in
\partial \Omega$. Assume that $\mu$ is an $\alpha$-Carleson
measure on $\Omega$. Then the $\kappa$-non-tangential maximal
function of an arbitrary Borel function $h:\Omega
\to [0,+\infty)$ satisfies
\begin{equation}\label{eq:IntIneqFromCarlesonNT}
\int_{\Omega} h^{\alpha p}\,d\mu \leq C \left(\int_{\partial
\Omega}
(\sup_{x\in \Gamma_{\Omega,\kappa}(\omega)} h(x))^p\,d\mathcal{H}^s\right)^{\alpha},\quad\text{for
all }0<p<\infty
\end{equation}
where $C$ depends on $p$, $\alpha$, $s$, $\kappa$,
$\gamma_{\alpha}(\mu)$, and the $s$-regularity constant of
$\partial \Omega$. If $\alpha=1$, then  $C$ can be chosen
independently of $p$.
\end{theorem}

We are in a position to formulate and prove the announced characterization of Carleson measures via the quasiregular mappings.

\begin{prop}[cf. Corollary 4.5 in~\cite{AK}]\label{obs:cor45}
   Let $f:B\to \Rn \setminus\{0\}$ be a $K$-quasiregular mapping in class~\eqref{est-min} such that $f$ satisfies the multiplicity condition~\eqref{cond-m} with $0\leq a<n-1$. Suppose further that $\mu$ is an $\alpha$-Carleson measure on $B$. Then it holds
   \begin{equation}\label{est:carleson}
    \int_{B} |f(x)|^{\alpha p}\,\ud \mu(x) \leq C \left(\int_{\mathbb{S}^{n-1}} |\tilde{f}(\om)|^p\,\ud \sigma(\om)\right)^{\alpha},\quad\text{for all }0<p<\infty,
   \end{equation}
  where $C$ depends on $n, K, a$ and the Carleson constant $\gamma_{\alpha}(\mu)$. Conversely, let $K\geq 1$ and $p>\frac{n-1}{K^{\frac{1}{n-1}}}$ be fixed and such that~\eqref{est:carleson} holds for all $K$-quasiregular mappings in class~\eqref{est-min} satisfying condition~\eqref{cond-m} with $0\leq a<n-1$. Then $\mu$ is an $\alpha$-Carleson measure. In particular, if $p>n-1$, the assertion holds for any $K$. 
 \end{prop}

\begin{proof}%[Proof of Proposition~\ref{obs:cor45}]
We apply Theorem~\ref{p:NontangentialInequality} for the metric space $(X, d)=(\R^n, |\cdot|)$, $\Om:=B$, function $h:=|f|$ and $s=n-1$ (note that $d\mathcal{H}^{n-1}\approx \ud \sigma$). Then, as implication (2) $\Rightarrow$ (3) of Theorem~\ref{thm-main} we apply Proposition~\ref{prop-zin} to conclude the proof of~\eqref{est:carleson}.

In order to show the opposite implication we use the same quasiconformal map as in the proof of~\cite[Corollary 4.5]{AK} and show that is satisfies the assumptions of the observation, see map $f$ below.

Since condition~\eqref{est:carleson} holds for all quasiregular maps in class~\eqref{est-min} satisfying the multiplicity condition~\eqref{cond-m} it is enough to find one map and show that it defines a Carleson measure. Set $\om\in\Sn$ and $0<r\leq 2$, also let $y:=(1+r)\om$ and $A$ be a sense reversing isometry of $\Rn$. We define the following map $f:B \to \Rn$:
\[
 f(x):=\frac{A(x-y)}{|x-y|^{1+K^{\frac{1}{1-n}}}}.
\]
Note that $|f(0)|=|y|^{-K^{\frac{1}{1-n}}}\not =0$ for any $y$ as defined above. Since $f$ is $K$-quasiconformal it trivially satisfies the multiplicity condition~\eqref{cond-m}. Moreover, the following discussion shows that $f$ also satisfies~\eqref{est-min} with exponent $\beta=K^{\frac{1}{1-n}}$ and $C=3^{\beta}$. Indeed, by the choice of $y$ we have that $1+r-|x|\leq |x-y|\leq 2+r-|x|$ and so, since 
\[
\frac{|f(x)|}{|f(0)|}= \left(\frac{|y|}{|x-y|}\right)^{\beta}
\]
it holds that
\[
\frac{1}{3^{\beta}} (1-|x|)^\beta \leq \frac{1}{(3-|x|)^\beta}\leq \left(\frac{1+r}{2+r-|x|}\right)^{\beta} \leq \left(\frac{|y|}{|x-y|}\right)^{\beta}\leq \left(\frac{1+r}{1+r-|x|}\right)^{\beta}\leq \frac{2^\beta}{(1-|x|)^\beta}.
\]
These inequalities are equivalent to~\eqref{est-min}, as trivially $1<1+r<2$ for $r=|x|$. Finally, we observe that for $p>\frac{n-1}{\beta}$ the above defined map $f$ satisfies 
\[
 \int_{\Sn} |f|^p\ud \sigma= \int_{\Sn} |A(x-y)|^p|x-y|^{-(1+\beta)p}\ud \sigma=\int_{\Sn} |x-\om-r\om|^{-p\beta}\ud \sigma \leq C r^{n-1-p\beta}.
\]
Moreover,
\[
\int_{B}|f(x)|^{p \alpha} \ud \mu \geq \int_{B\cap B(\om, r)}|x-y|^{-p \alpha \beta} \ud \mu \geq C^{-1} r^{-p \alpha\beta} \mu(B\cap B(\om, r)),  
\]
where $C$ is independent of $\om$ and $r$. By combining the above two estimates and applying condition~\eqref{est:carleson} we get that $\mu$ is an $\alpha$-Carleson measure on $B$.
\end{proof}

We next present the following two consequences of Proposition~\ref{obs:cor45}.

For a given map $f:B\to \Rn$ we define
\[
 M(r,f):=\sup_{\{x\in B: |x|=r|\}} |f(x)|,\quad 0<r<1.
\]

\begin{cor}[cf. Theorem 3.3 in~\cite{AK}]
    Let $f:B\to \Rn \setminus\{0\}$ be a $K$-quasiregular mapping in class~\eqref{est-min} such that $f$ satisfies the multiplicity condition~\eqref{cond-m} with $0\leq a<n-1$. If for some $0<p<n$ it holds that $f\in \mathcal{H}^p$, then
 \begin{equation}
  \int_{0}^{1}(1-r)^{n-2} M(r,f)^p\,\ud r<\infty.
 \end{equation}
\end{cor}
\begin{proof}
 The proof follows the lines of the proof of \cite[Theorem 3.3]{AK} and by application of Proposition~\ref{obs:cor45}. Finally, implication (1) $\Rightarrow $ (2) in Theorem~\ref{thm-main} implies the assertion of the corollary.
\end{proof}

\begin{cor}[cf. Corollary 4.5 in~\cite{AK}]
  Let $f:B\to \Rn \setminus\{0\}$ be a $K$-quasiregular mapping in class~\eqref{est-min} such that $f$ satisfies the multiplicity condition~\eqref{cond-m} with $0\leq a<n-1$.  Suppose that $L\subset \R^n$ be an $(n-1)$-dimensional plane containing the origin. Then
\[
 \int_{L\cap B}|f(x)|^p\ud H^{n-1}\leq C(n, p, K, a) \int_{\Sn} |f(\om)|^p\ud \sigma,\quad 0<p<n,
\] 
where $H^{n-1}$ denotes the $(n-1)$-Hausdorff measure.
\end{cor}
\begin{proof}
 Let us observe that the measure defined as $\ud \mu:=\ud H^{n-1}$ on $L\cap B$ and $0$ otherwise is the $1$-Carleson measure on $B$. Indeed, if $\om\in \Sn$ and $r>0$, then $\mu(B\cap B(\om, r))=\mu((B\cap L)\cap B(\om, r)) \leq H^{n-1}(L \cap B(\om, r))\leq C r^{n-1}$. By applying Proposition~\ref{obs:cor45} we obtain the claim of the corollary.
\end{proof}

\subsection*{Applications of Theorem~\ref{thm-main} for quasiregular mappings}

The main theorem of our work, i.e. Theorem~\ref{thm-main}, together with the auxiliary results leading to it, have several 
consequences for the theory of quasiregular maps. The results presented below are counterparts of quasiconformal results in ~\cite{AK}.

\begin{lem}[cf. Lemma 5.3 in~\cite{AK}]\label{lem-48}
   Let $n \geq 2$ and $f:B\to \Rn \setminus\{0\}$ be a $K$-quasiregular mapping in class~\eqref{est-min} satisfying the multiplicity condition~\eqref{cond-m} with $0\leq a<n-1$. Then
   \[
    f\in \mathcal{H}^p \quad\hbox{if and only if}\quad \int_{B}|f(x)|^{p-1} |Df(x)|\,\ud x <\infty.
   \]
\end{lem} 
\begin{proof}
 Note that 
 $$
 \int_{B}|f(x)|^{p-1} |Df(x)|\,\ud x= \int_{B}|f(x)|^{p} \frac{|Df(x)|}{|f(x)|}\,\ud x.
 $$
 By Lemma~\ref{lem-Jones-qr}, the measure $\ud \mu(x)\frac{|Df(x)|}{|f(x)|}\,\ud x$ is the $1$-Carleson measure. This combined with Proposition~\ref{obs:cor45} and implication (1) $\Rightarrow $ (2) in Theorem~\ref{thm-main} implies the sufficiency part of the assertion. Moreover, the proof of the necessity part follows exactly the same lines as in the corresponding proof of Lemma 5.3 in~\cite{AK} and, therefore, we omit it.
\end{proof} 

Next observation shows that the integral measure of oscillations vanishes for quasiregular maps as in the main Theorem~\ref{thm-main}.
\begin{cor}[cf. Corollary 4.4 in~\cite{AK}]
   Let $n \geq 2$ and $f:B\to \Rn \setminus\{0\}$ be a $K$-quasiregular mapping in class~\eqref{est-min} satisfying the multiplicity condition~\eqref{cond-m} with $0\leq a<n-1$. If $\tilde{f}\in L^{p}(\Sn)$, then
   \[
   \lim_{r\to 0} \int_{\mathbb{S}^{n-1}} |f(r\om)-f(\om)|^p\,\ud \sigma=0.
   \]
\end{cor} 
\begin{proof}
As in~\cite{AK} we note that, by the definition, it holds $|f(r\om)-f(\om)|\leq 2f^{*}(\om)$. Then, implication (2) $\Rightarrow $ (3) in Theorem~\ref{thm-main} together with the dominated convergence theorem imply the assertion.
\end{proof}

 Note that alternatively, we may assume the growth condition~\eqref{cond-g} instead of  the $L^p$-integrability of the non-tangential limit function. This follows from Theorem 4.1 in~\cite{kmv}.

Finally, we apply the characterization of Carleson measures on $B$ to relate Hardy spaces and quasiregular mappings integrable on the unit ball (a counterpart of the Bergman spaces), thus partially generalizing Theorem 9.1 in~\cite{AK}. Namely, our approach allows us to show that mappings in Hardy space $\mathcal{H}^p$ are globally integrable with some power. However, in order to show the opposite implication our methods are insufficient due to the lack of good modulus inequality for quasiregular mappings with unbounded multiplicity, cf. the proof of implication (3.5) to (3.4) in Theorem 3.3 in \cite{AK}. Instead, we show the weaker condition, see below.

For a map
$f:B\to \Rn$ and $0<p<\infty$ we define
 \[
  \|f\|_{A^p}:=\left(\int_{B} |f(x)|^p\,\ud x\right)^{\frac1p},
 \]
 and write $f\in A^p$ if  $\|f\|_{A^p}<\infty$.

\begin{theorem}\label{t: Ap-spaces}
   Let $f:B\to \Rn \setminus\{0\}$ be a $K$-quasiregular mapping in class~\eqref{est-min} such that $f$ satisfies the multiplicity condition~\eqref{cond-m} with $0\leq a<n-1$.
\begin{itemize}
\item[(1)] If $f\in \mathcal{H}^p$, then $f\in A^{\frac{pn}{n-1}}$.
\item[(2)] If $f\in A^p$, then $\int_{0}^{1} (1-r)^{n-2} M(r,f)^{p'}\ud r<\infty$ for all $p'<p\frac{n-1}{n}$. 
\end{itemize}
\end{theorem}

\begin{proof}[Proof of Theorem~\ref{t: Ap-spaces}]
Let $f\in \mathcal{H}^p$ be as in assumptions of the theorem. Since the Lebesgue measure is $n$-Ahlfors regular, we obtain
by~\eqref{eq:CarlesonCOnstAbstr} for $s=n$ that its restriction to $B$ is $\alpha$-Carleson for $\alpha=\frac{n}{n-1}$ and hence $f\in A^{\frac{pn}{n-1}}$ by Proposition~\ref{obs:cor45}.

In order to show assertion (2), we appeal to the supremum estimate~\eqref{sup-est} applied to coordinate functions of $f$ on hyperbolic balls for $q=p$ to get that
\[
 |f(x)|^p\leq \frac{C}{(1-|x|)^n}\int_{B\left(x,\frac12(1-|x|)\right)} |f(y)|^p\,\ud y\leq \frac{C \|f\|^p_{A^p}}{(1-|x|)^n}
\]
at each point $x\in B$, as $f\in A^p$. Therefore, by the direct integration we obtain the assertion:
\[
 \int_{0}^{1} (1-r)^{n-2} M(r,f)^{p'}\ud r \lesssim_{C,\,\|f\|_{A^p}} \int_{0}^{1} (1-r)^{n-2} (1-r)^{-p'\frac{n}{p}}\ud r<\infty,
\]
provided that $p'<p\frac{n-1}{n}$.
\end{proof}

\subsection*{Theorem~\ref{thm-main} and $\mathcal{A}$-harmonic functions}

The nonlinear potential theory influences quasiregular world significantly, as we already seen in Lemma 2.1 and in the proof of Observation~\ref{thm1}. In this short section we would like to address how the results of our work for quasiregular maps affect the corresponding $\A$-harmonic equations. 

Let us recall that given the $\A$-harmonic operator of type $n$, i.e. satisfying conditions (1)-(3) in Section 2.2 for $p=n$, one shows that $f^{\#}\A$, the pull-back of $\A$ under a quasiregular map $f$, is also of type $n$, see Definition in Section 14.35 and Lemma 14.38 in~\cite{hkm}. It follows, that given a non-constant quasiregular mapping $f:B\to \Rn \setminus\{0\}$, the function $v(x):=\ln|f(x)|$ is an $f^{\#}\A$-harmonic function in $B\setminus f^{-1}(0)$, where $\A$ is the $n$-harmonic operator $\A(\xi)=|\xi|^{n-2}\xi$ and $\alpha=\frac 1K$, $\beta=K$, see~\cite[Lemma 14.19]{hkm}. If $f$ satisfies~\eqref{est-min}, then by direct computations we find that $v$ satisfies the following growth condition: 
\[
 |v(x)-v(0)|\leq \ln C+\alpha \ln \left(\frac{1+|x|}{1-|x|}\right), 
\]
where constants $C$ and $\alpha$ are as in \eqref{est-min}. In particular, oscillations of $v$ are bounded for small enough $|x|$. Moreover, the assertions of Theorem~\ref{thm-main} now read as follows: 
\[
 e^v\in \mathcal{H}^p \Leftrightarrow \widetilde{e^v} \in L^p(\Sn) \Leftrightarrow (e^v)^* \in L^p(\Sn).
\]
Furthermore, the assertion of Lemma~\ref{lem-48} can be equivalently expressed as follows:
\[
 f\in \mathcal{H}^p \quad \Leftrightarrow \quad \int_{B} |\nabla v(x)| e^{p v(x)}\,\ud x <\infty.
\]

Moreover, in dimension $2$ an $\A$-harmonic function may define a quasiregular mapping via the complex gradient, i.e. $f:=u_z=u_x-iu_y$ is quasiregular. For example, this is the case of $p$-harmonic equation with $K=K(p)$, also $\A=\A(x)$ under ellipticity assumption, as well as $\A=\A(\nabla u)$ for $\delta$-monotone $\A$ and appropriate integrability assumptions on $u$, see Chapters 16.1-16.5 in~\cite{aim}. 

For the sake of simplicity of the presentation, let us assume that $u:B\to \R$ is a $p$-harmonic function for some $1<p<\infty$. Then $K=\frac12\left(p-1+\frac{1}{p-1}\right)$ and the Miniowitz estimate~\eqref{est-min} reads:
\[
\frac{1}{C}\left(\frac{1-|z|}{1+|z|}\right)^{\alpha} \leq \frac{|u_z(z)|}{|u_z(0)|} \leq C\left(\frac{1+|z|}{1-|z|}\right)^{\alpha},
\]
where $\alpha=2K$ and $C=2^{16K}$. Moreover, since $f$ needs to omit some set containing the origin, we have that $u_z\not=0$ in $B$ and so $u$ is in fact real analytic in $B$, see~\cite{man}. 
%As for the multiplicity condition~\eqref{cond-m} on $f$, it is harder to see what it says about $u$. \komT{I don't know how to translate condition (M) for $f=|u_z|$ to a property of $u$?}

We can express assertions of Theorem~\ref{thm-main} in terms of $u$ and its gradient as follows:
\[
 u_z\in H^q \Leftrightarrow \widetilde{u_z} \in L^q(\Sn) \Leftrightarrow (u_z)^* \in L^q(\Sn),\quad 0<q<\infty.
\]
In particular, the above equivalence is to our best knowledge first characterization of non-tangential maximal functions in the setting of the $p$-harmonic functions.

\section*{Appendix}

The purpose of this section is to prove Lemma~\ref{lem:Tx} about the properties of conformal map $T_x$, see~\eqref{def:Tx}. For the readers convenience we recall the assertions of the lemma:
\begin{align*}
&\hbox{(1) }\,\,J_{T_{x}}\approx \frac{1}{r^n} \quad\hbox{on }B\cap B(\om,r),\qquad
\hbox{(2) }\,\,\frac{1-|T_{x}(y)|}{1-|y|} \approx \frac{1}{r}\quad \hbox{for }y\in B\cap B(\om,r),\\
&\hbox{(3)}\,\,B\left(0, \frac{k-k^2}{(2+k)^2}\right)\subset T_{x}\Big(B(x,k(1-|x|)\Big)\subset B\left(0,k^2+ 2k\right) \hbox{ for }\,0<k<\sqrt{2}-1.
\end{align*}
\begin{proof}[Proof of Lemma~\ref{lem:Tx}] By Formulas (24)-(26) in~\cite{ahl}, map $T_x$ can be equivalently expressed as follows:
\begin{equation}
 T_x(y)=({\rm Id}-2Q(x))(x^*+(|x^*|^2-1)\frac{y-x^*}{|y-x^*|^2},
\end{equation}
where $x^*=\frac{x}{|x|^2}$ stands for a reflection of point $x$ in the unit sphere (cf.~\cite[Formula (10)]{ahl}) with convention that $0^*=\infty$, while $Q(x)$ is a square $n\times n$ matrix with entries $Q(x)_{ij}=\frac{x_ix_j}{|x|^2}$ for $i,j=1,\ldots,n$. By \cite[Formula (14)]{ahl} upon direct but tedious computations we get that the Jacobian of $T_x$ satisfies:
\[
J_{T_x}(y)=\det({\rm Id}-2Q(x))\,\det\Big(D\left((x^*+(|x^*|^2-1)\frac{y-x^*}{|y-x^*|^2}\right)\Big)=\left(\frac{1-|x|^2}{|x|^2}\right)^n\frac{1}{|y-x^*|^{2n}}.
\]
Choose $x\in B\cap B(\om, r)$ such that $x$ lies on the radial segment joining the origin with $\om$ and $|x-\om|=\frac{r}{2}$. Then $|x|=1-\frac{r}{2}$ and $y\in B\cap B(\om, r)$
\[
|y-x^*|\geq d(y, \partial B)\geq 1-|y|\geq 1-r.
\]
Similarly, we find that
\[
 |y-x^*|\leq |y-x|+|x-x^*|\leq r+1-|x|+1-|x^*|=1+\frac32 r-\frac{1}{1-\frac{r}{2}}.
\]
If $r$ is small enough, say $r<\frac12$, then $|x|\geq \frac34$ and $\frac{1-|x|^2}{|x|^2}\leq \frac43 r$, whereas $(1-r)^{-2n}\leq r^{-2n}$. Hence, $J_{T_x}\lesssim r^{-n}$ on $B\cap B(\om, r)$. The analogous estimates gives us also the lower bound and, thus, assertion (1) is proven.

In order to show assertion (2), we use \cite[Formula (33)]{ahl} to obtain that
\[
 \frac{1-|T_{x}(y)|}{1-|y|}=\frac{1-|T_{x}(y)|^2}{1+|T_{x}(y)|}\frac{1}{1-|y|}=\frac{1-|x|^2}{|x|^2} \frac{1-|y|}{|y-x^*|^2} \frac{1}{1+|T_{x}(y)|} \approx r \frac{1-|y|}{|y-x^*|^2},
\]
where in the last step we use also the bounds on $|x|$ from the previous part of the proof. Again, by the previously obtained estimates for $|y-x^*|$ we get, as $r<\frac12$, that
\[
\frac{1}{r^2}\lesssim\frac{1-r}{(1+\frac32 r-\frac{1}{1-\frac{r}{2}})^2}\leq \frac{1-|y|}{|y-x^*|^2}\leq \frac{1}{(1-r)^2}\leq \frac{1}{r^2}.
\]
The first estimate on the left-hand side follows from the inequality
\[
\frac{r^2(1-r)}{(1+\frac32 r-\frac{1}{1-\frac{r}{2}})^2}\geq \frac12 \left(\frac{r}{1+\frac32 r-\frac{1}{1-\frac{r}{2}}}\right)^2 \geq \frac12 \left(\frac{r(1+\frac32 r)}{(1+\frac32 r)^2-1}\right)^2\geq \frac12 \left(\frac23 \frac{1+\frac32 r}{2+\frac32 r}\right)^2 \geq \frac12 \left(\frac{8}{33}\right)^2.
\]
Finally, we prove assertion (3) of Lemma~\ref{lem:Tx}. Let us assume that $y\in B(x, k(1-|x|))$, for $0<k<\sqrt{2}-1$. We first estimate the denominator of $T_x$, cf.~\eqref{def:Tx}. On one hand we have that:
$$
|x|~\Big|y-\frac{x}{|x|^2}\Big|=\Big|\frac{x |x|}{|x|^2}-y|x|\Big|\geq 1-|y||x|\geq 1-|x|.
$$
On the other hand it holds:
\begin{align*}
|x|~\Big|y-\frac{x}{|x|^2}\Big|&=\Big|y|x|-\frac{x |x|}{|x|^2}\Big|\leq \Big|y|x|-x|x|\Big|+\Big|x|x|-\frac{x|x|}{|x|^2}\Big|\\
&\leq |y-x|+(1-|x|^2),
\end{align*}
as $\Big|x|x|-\frac{x |x|}{|x|^2}\Big|=\dist(x|x|, \partial B)=1-|x|^2$. Therefore, for any $y\in B(x, k(1-|x|))$, we have
$$
|x|~\Big|y-\frac{x}{|x|^2}\Big|\leq k(1-|x|)+2(1-|x|)=(2+k)(1-|x|).
$$
Next, let us estimate  the numerator of $T_x(y)$ for $y\in B(x, k(1-|x|))$:
\begin{align*}
|(1-|x|^2)(y-x)-|y-x|^2x|&\leq (1-|x|^2)|y-x|+|y-x|^2 \\
&\leq 2k(1-|x|)^2+k^2 (1-|x|)^2\\&=(2k+k^2)(1-|x|)^2.
\end{align*}
\bigskip
Furthermore, if $|x-y|=k(1-|x|)$, then we get
\begin{align*}
|(1-|x|^2)(y-x)-|y-x|^2 x|&\geq 
(1-|x|^2)|y-x|-|y-x|^2 |x|\\&\geq k(1-|x|^2)-k^2(1-|x|)^2\\&
=(k-k^2)(1-|x|)^2.
\end{align*}
These estimates show that
$$
B\left(0, \frac{k-k^2}{(2+k)^2}\right)\subset T_{x}\left(B(x,k(1-|x|)\right)\subset B\left(0, k^2+2k\right).
$$
By direct computations one checks that $T_x^{-1}=T_{-x}$, and hence the analogous inclusions hold for the inverse map $T_x^{-1}$. Thus, the proof of assertion (3) is completed.
\end{proof}

%\textit{Tomasz Adamowicz:} Institute of Mathematics, Polish Academy of Sciences, ul. \'Sniadeckich 8, Warsaw 00-656, Poland. E-mail address: tadamowi@impan.pl
%
%\textit{Mar\'ia J. Gonz\'alez:} Departamento de Matem\'aticas, Universidad de C\'adiz, 11510 Puerto Real (C\'adiz), Spain. E-mail address: majose.gonzalez@uca.es

\end{document}